\documentclass[11pt,leqno]{article}
\usepackage[utf8]{inputenc}

\usepackage{amsmath}
\usepackage{amsthm}
\swapnumbers

\usepackage{enumerate}
\usepackage{mathtools}

\usepackage[colorlinks=false]{hyperref}

\usepackage[charter]{mathdesign}

\usepackage{xy}
\xyoption{all}

\numberwithin{equation}{section}

\theoremstyle{plain}
    \newtheorem{theorem}[equation]{Theorem}
    \newtheorem{lemma}[equation]{Lemma}
    
    \newtheorem{proposition}[equation]{Proposition}

    \newtheorem*{theorem*}{Theorem}
    \newtheorem*{proposition*}{Proposition}
    \newtheorem*{corollary*}{Corollary}
    \newtheorem*{lemma*}{Lemma}
    \newtheorem*{conjecture*}{Conjecture}
    \newtheorem{definition-theorem}[equation]{Definition/Theorem}
    \newtheorem{definition-lemma}[equation]{Definition/Lemma}
\theoremstyle{definition}
    \newtheorem{definition}[equation]{Definition}
    \newtheorem{example}[equation]{Example}
    
    \newtheorem{notation}[equation]{Notation}
    
    \newtheorem{remark}[equation]{Remark}

    \newcommand{\R}{\mathbb{R}}
    \newcommand{\C}{\mathbb{C}}
    
    \newcommand{\Z}{\mathbb{Z}}

    \newcommand{\U}{\mathfrak{U}}
    
    \newcommand{\Fourier}{\mathcal{F}}
    
   	\renewcommand{\phi}{\varphi}
	\let\epsilon\varepsilon

\newcommand{\llangle}{\langle\!\langle}
\newcommand{\rrangle}{\rangle\!\rangle}

\newcommand{\dd}{d}

\newcommand{\germ}{\mathfrak}

\newcommand{\id}{\mathrm{id}}

    \DeclareMathOperator{\Hom}{Hom}
    \DeclareMathOperator{\End}{End}

    \DeclareMathOperator{\Res}{Res}

    \DeclareMathOperator{\Aut}{Aut}
    
    \DeclareMathOperator{\Ind}{Ind}
    \DeclareMathOperator{\ind}{ind}
    \DeclareMathOperator{\SL}{SL}

	\DeclareMathOperator{\vol}{vol}

	\DeclareMathOperator{\HC}{\mathcal C}
	\DeclareMathOperator{\Sc}{\mathcal S}
	
	\DeclareMathOperator{\conv}{conv}

	\DeclareMathOperator{\Fr}{{\text{\sl Restr}\,}}
	\DeclareMathOperator{\Extn}{\text{\sl Extn\,}}
	
	\newcommand{\even}{{\textnormal{even}}}
	\newcommand{\odd}{{\textnormal{odd}}}

\newcommand{\op}[1]{{#1}_{\textnormal{--}}}
\newcommand{\plus}[1]{{#1}_{\textnormal{+}}}
\newcommand{\pnm}[1]{{#1}_{{\pm}}}
\newcommand{\mnp}[1]{{#1}_{\mp}}

	\newcommand{\SFMod}[1]{\mathsf{SFMod}_{#1}}
	\newcommand{\SF}{\mathsf{SF}}

\title{A Second Adjoint Theorem for $\SL(2,\R)$}

\author{Tyrone Crisp\thanks{Partially supported by the Danish National Research Foundation through the Centre for Symmetry and Deformation (DNRF92).} \and 
Nigel Higson\thanks{Partially supported by the US National Science Foundation through the grant DMS-1101382.}}

\AtEndDocument{{ 
  \bigskip
  \footnotesize

\noindent \textsc{Max-Planck-Institut für Mathematik, Vivatsgasse 7, 53111 Bonn, Germany.}\par\nopagebreak
\noindent \href{mailto:tyronecrisp@mpim-bonn.mpg.de}{tyronecrisp@mpim-bonn.mpg.de}

  \medskip

\noindent \textsc{Department of Mathematics, Penn State University, University Park, PA 16802, USA.}\par\nopagebreak
\noindent \href{mailto:higson@math.psu.edu}{higson@math.psu.edu}
}}

\date{}

\begin{document}

\maketitle

\begin{abstract}
 \noindent We formulate a second adjoint  theorem in the context of tempered representations of  {real} reductive   groups, and prove it in the case of $\SL(2,\R)$.\end{abstract}

\section{Introduction}

Bernstein's famous \emph{second adjoint theorem} in the smooth representation theory of reductive $p$-adic groups asserts that for every parabolic subgroup $P$ of a reductive $p$-adic algebraic group $G$, the functor of parabolic induction has not only a left adjoint (this is Frobenius reciprocity) but also a right adjoint.   See \cite{Bernstein-2-adjoint,Bernstein-p-adic}.  The purpose of this paper is to formulate and prove a similar theorem in the context of \emph{real} reductive groups, but specifically for \emph{tempered} representations.  We shall concentrate on the group $G=\SL(2,\R)$; much greater generality is possible, but at the cost of  complicating the discussion.

We shall approach representations in general through convolution algebras, and tempered representations in particular through Harish-Chandra's Schwartz space  $\HC(G)$.  By a tempered representation of $G$ we shall mean a smooth, Fr\'echet module (an $\SF$-module) over $\HC(G)$; see Section~\ref{sec-categories}.   

Denote by $\plus{N}$ the group of unipotent upper triangular matrices in $G= \SL(2,\R)$,  by $\plus{P}$ the associated parabolic subgroup, and by $L$ its Levi factor (the diagonal matrices).  

\begin{theorem*}
The functor of parabolic induction
\[
\Ind_{\plus{`P}}^G \colon \SFMod{\HC(L)}\longrightarrow  \SFMod{\HC(G)}
\]
has both a left adjoint and a right adjoint.
\end{theorem*}

We refer the reader to Sections~\ref{sec-frobenius} and \ref{sec-bernstein} for the precise formulation.  It  is worth emphasizing  here, however, that our functor of parabolic induction is the standard one in the context of representations on Fr\'echet spaces; see Proposition~\ref{Ind_ind_proposition}.

The adjoint functors that arise in our second adjoint theorem are  the same as those studied by Bernstein, namely \emph{parabolic restriction} with respect to $\plus{P}$ in the case of the left adjoint, and parabolic restriction with respect to the opposite parabolic subgroup $\op{P}$ in the case of the right adjoint.  

Moreover the adjunction isomorphisms are defined in the same way as those studied by Bernstein: the counit transformation for the first adjunction  (Frobenius) is associated to the canonical inclusion of $L$ as a closed subset of the double coset space $\plus{N}\backslash G / \plus{N}$, while the unit transformation for  the second adjunction is associated to the canonical inclusion of $L$ as an \emph{open} subset of $\op{N} \backslash G / \plus{N}$.   

The unit transformation for Frobenius reciprocity is straightforward too, and is obtained directly from the formula for the   action of $G$ on   parabolically induced representations. (The unit transformation is perhaps best understood from the perspective of operator algebras, which was our starting point.   See \cite{CH_cb}.)  

However our approach to the \emph{counit} transformation for the second adjoint   departs from Bernstein's.  We shall need to make use of the theory of the standard intertwining integral, and also Harish-Chandra's theory of wave packets.  So the proof is not at all elementary.  And whereas everything else in the paper extends easily to general real reductive groups and parabolic subgroups, here we shall concentrate on the special case where  $G=\SL(2,\R)$.  The necessary results are stated in Section~\ref{sec-proof} and   proofs  are given in  Section~\ref{appendix-fourier}.   

Although we have not yet attempted a proof in the general case,  it seems likely to us that  the approach we follow for $\SL(2,\R)$ can be extended to all $G$, and all parabolic subgroups.  But, at the very least, a substantial amount of Harish-Chandra's theory will be required.

The recent work of Bezrukavnikov and Kazhdan \cite{BezKaz} offers a geometric perspective on the second adjoint theorem in the $p$-adic case.   There is a related approach in the tempered real case that involves the wave equation defined by the Casimir operator. We shall present this elsewhere.   

Another ongoing project  is to reorganize some of the foundational discoveries of Harish-Chandra, Langlands and others about tempered representations around  the second adjoint theorem and its consequences.   Once again we aim to present this work elsewhere.

\section{Categories of SF-representations}
\label{sec-categories}

By a \emph{Fr\'echet algebra} we shall mean a Fr\'echet space $\mathcal{A}$ that is equipped with a bilinear, continuous and associative multiplication operation.  Recall that in the context of Fr\'echet spaces, separately continuous bilinear maps are automatically jointly continuous.  

We shall denote by $\otimes$ the completed projective tensor product of Fr\'echet spaces (actually, all the spaces that we shall calculate with will be both Fr\'echet and nuclear, and for these the choice of tensor product is immaterial).  If $W$ is a Fr\'echet space and a  right module over an algebra $\mathcal{A}$, and $V$ is a Fr\'echet space and a left module over $\mathcal{A}$, then we shall denote by $W \otimes _{\mathcal{A}} V$ the quotient of $W \otimes V$ by the \emph{closed} subspace spanned by the relators 
\[
wa \otimes v - w \otimes av\in W \otimes V.
\]
It is a Fr\'echet space in its own right. 

\begin{definition} 
Let $\mathcal{A}$ be a Fr\'echet algebra.  A (left)  \emph{$\SF$-module} over $\mathcal{A}$ is a Fr\'echet space $V$ that is equipped with a continuous action of $\mathcal{A}$ for which the   map 
\[
\mathcal{A} \otimes _{\mathcal {A}} V \longrightarrow V
\]
induced from the module action is an isomorphism.  A \emph{morphism} of $\SF$-modules is a continuous map of Fr\'echet spaces that is also an $\mathcal{A}$-module map. We shall denote by 
$\SFMod{\mathcal{A}}$ the category of $\SF$-modules over $\mathcal{A}$.
\end{definition}

\begin{remark} 
In the guiding context of $p$-adic groups, the relevant convolution algebra $\mathcal H (G)$  of locally constant, compactly supported functions on a $p$-adic group $G$ is obviously not a Fr\'echet algebra, but we may use the algebraic tensor product, and the smooth representations of $G$ in the sense of \cite[Section 1.1]{Bernstein-p-adic} are precisely those for which the natural map 
\[
\mathcal{H}(G) \otimes _{\mathcal {H}(G)} V \longrightarrow V
\]
is an isomorphism.
\end{remark}
 
 Now, let $G$ be the group of real points of a connected linear  reductive algebraic group defined over $\R$ (in brief, a \emph{real reductive group} from now on).  Its algebraic-geometric  structure gives $G$ the structure of an (affine)
Nash manifold in the sense of \cite{AizGour}, and so there is a canonical  associated space $\Sc(G)$ of \emph{Schwartz functions} on $G$ \cite[Section 4]{AizGour}.   Since we are principally interested in the group $SL(2,\R)$, let us describe  the structure on $\Sc(G)$ explicitly  in this case.

\begin{definition} \label{Sc_definition}
Let $G=SL(2,\R)$.  The \emph{Schwartz space}  $\Sc(G)$ is the space of smooth, complex valued functions $f$ on $G$ for which 
\begin{equation}\label{Sc_seminorms}
\sup_{g\in G} | (X  f )(g)  | <\infty 
\end{equation}
for every polynomial differential operator $X$ on $G$. (A polynomial differential operator is a linear partial differential operator on $G$ that preserves the subspace of functions that are polynomials in the matrix entries of $g\in G$.)
\end{definition}

\begin{remark}
This definition is equivalent to the one appearing in \cite[Section 7]{Wallach1}, where $\Sc(G)$ is called the  \emph{space of rapidly decreasing functions}, and  to the definitions in \cite[p.\ 56]{BK} and \cite[p.\ 392]{Casselman}, where $\Sc(G)$ is called the  {Schwartz space}, as above. 
\end{remark}

The seminorms appearing in \eqref{Sc_seminorms} make $\Sc(G)$ into a nuclear Fr\'echet space and   a Fr\'echet algebra under convolution.  See \cite[Section 7.1]{Wallach1} and \cite[Section 2]{BK}. 
The following two propositions describe the $\SF$-modules over $\Sc(G)$.

\begin{proposition}[{\cite[Proposition 2.20]{BK}}]
\label{prop-moderate-growth}
Let $V$ be a Fr\'echet space  equipped with  a continuous  action of $\Sc(G)$. The following conditions on $V$ are equivalent:
\begin{enumerate}[\rm (a)]
\item $\Sc(G)V =V$
\item  There is a unique continuous $G$-action on $V$ with the properties that 
\begin{enumerate}[\rm (i)]
\item  For every continuous seminorm $p$ on $V$, the function $g\mapsto p(g{\cdot} v) $ is bounded by a polynomial in the matrix entries of $g$, independent of $v$, times $q(v)$, where $q$ is a second continuous seminorm on $V$ \textup{(}that is, the action has \emph{moderate growth} in the sense of Casselman \textup{\cite{Casselman}}\textup{)}.
\item The action of $\Sc(G)$ on $V$ is given by the integral formula 
\[
f\cdot v = \int _G f(g) \, g{\cdot} v \, \dd g
\]
\textup{(}the integral converges in view of \textup{(i)}\textup{)}.
\item For every $v\in V$ the map $g\mapsto g{\cdot}v$ is smooth. \qed
\end{enumerate}

\end{enumerate}
\end{proposition}

\begin{example} The left and right actions of $G$ on the  Fr\'echet space $V=\Sc(G)$ satisfy the conditions in item (b) above (and the condition $\Sc(G)V = V$ follows, for example, from the Dixmier-Malliavin theorem \cite{DixmierMalliavin}).
\end{example}

\begin{proposition}
\label{prop-SF-equals-SF}
If $V$ is a Fr\'echet space that is equipped with a continuous action of $\Sc(G)$, then $V$ is an $\SF$-module over $\Sc(G)$ if and only if $\Sc(G)V = V$.
\end{proposition}

\begin{proof} 
If $V$ is an $\SF$-module, then it is a quotient of $\Sc(G)\otimes V$.  The tensor product satisfies the conditions in item (b) of Proposition~\ref{prop-moderate-growth}, and the class of representations satisfying these conditions is closed under quotients \cite[Lemma 2.9]{BK}, so $V$ satisfies the conditions too. 

Suppose conversely that $\Sc(G)V = V$.  The following argument, taken from     \cite[Proposition 1.4]{BlancBryl},   constructs a chain contraction  of the so-called $b'$-com\-plex
\begin{equation}
\label{eq-b-prime-cplx}
\cdots \longrightarrow \Sc(G)\otimes\Sc( G)\otimes  V \longrightarrow \Sc(G) \otimes   V \longrightarrow V  
\end{equation}
with differentials
\[
f_1\otimes \cdots \otimes f_p\otimes v \longmapsto \sum_{j=1}^p (-1)^{j-1} f_1 \otimes \cdots \otimes f_jf_{j+1} \otimes \cdots \otimes v
\]
(in the formula  we set $f_{p+1} = v$).   
At the bottom level, the chain contraction establishes  the isomorphism 
\[
 \Sc(G)\otimes_{\Sc(G)} V\stackrel \cong \longrightarrow V
 \]
that we  require.

Because $\Sc(G)$ is a nuclear Fr\'echet space, the chain groups in \eqref{eq-b-prime-cplx}  identify with the spaces 
\[
\Sc(G\times\cdots \times G,V)
\]
 of $V$-valued Schwartz functions on $G{\times}\cdots {\times} G$ (to define the concept of $V$-valued Schwartz function, replace the absolute value in \eqref{Sc_seminorms} with any of the continuous seminorms on $V$).  
The contraction operators are defined by 
\[
f\longmapsto \Bigl [ (g_0,\dots, g_p) \mapsto u(g_0)f(g_0g_1,\dots , g_p)\Bigr ]
\]
where $f\in \Sc(G\times\cdots \times G,V)$, and where $u$ is a smooth, compactly supported function on $G$ with total integral $1$.
\end{proof}

\section{Tempered Representations}
\label{sec-tempered-rep}

The focus of our attention in this paper will be a second, also well known, convolution algebra: the Schwartz space $\HC(G)$ of Harish-Chandra. In this section we shall recall the definition of $\HC(G)$ in the case of $G=\SL(2,\R)$. We refer the reader to \cite{Wallach1} for the general case. 

Let  $G=\SL(2,\R)$, and denote by  $\plus{P}$ and $\op{P}$   the parabolic subgroups of upper- and lower-triangular matrices, respectively. Denote by $\plus{N}$ and $\op{N}$ their respective unipotent radicals of unipotent upper- and lower-triangular matrices, and let $L=\plus{P}\cap \op{P}$ be the common Levi subgroup of diagonal matrices.

Let $\pnm{\delta}: L\to \R^+$ be the homomorphisms characterized by the equalities 
\begin{equation}\label{delta_equation}
\int_{\pnm{N}} f(n)\, \dd n = \pnm{\delta}(\ell) \int_{\pnm{N}} f(\ell n \ell^{-1})\, \dd n
\end{equation}
for all $f\in C_c^\infty(\pnm{N})$. Explicitly, 
\[
\pnm{\delta}: \begin{bmatrix} \alpha & 0 \\ 0 & \alpha^{-1} \end{bmatrix} \mapsto \alpha^{\pm2}.
\]
Denote by $K\subseteq G$ the maximal compact subgroup $\operatorname{SO}(2)$ of rotation matrices in $G$, and  denote by    $A\subseteq G$ be the  positive diagonal matrices. 
Extend $\pnm{\delta}$ to maps on $G$ via the \emph{Iwasawa decompositions} $G= K A \pnm{N}$:
\begin{equation}
\label{eq-delta-def2}
\qquad \qquad 
\pnm{\delta}(kan) \coloneqq \pnm{\delta}(a) \qquad (k\in K,\ a\in A,\ n\in \pnm{N}).
\end{equation}

\begin{definition} \label{Xi_definition}
The \emph{Harish-Chandra $\Xi$-function} on $G$ is defined by the integral formula
\[
\Xi_G(g) \coloneqq \frac{1}{\vol(K)} \int _K \pnm{\delta} (gk)^{-1/2} \, \dd k.
\]
(Both choices of $\pnm{\delta}$ give the same function $\Xi_G$. The $\Xi$-function does however depend on the choice of maximal compact subgroup $K$.)\end{definition}
 
The most important properties of the $\Xi$-function are that it is a spherical function, 
\begin{equation}
\label{eq-spherical-fn}
\Xi_G(g_1)\Xi_G(g_2) =\frac{1}{\vol(K)}  \int _K \Xi_G(g_1kg_2) \, \dd k ,
\end{equation}
and that it is almost an $L^2$-function.  The latter property is made precise as follows.

\begin{definition}\label{scale_definition}
Denote by $\|\,\,\|:A\to [1,\infty)$ the function 
\[
\left\| \left [\begin{smallmatrix} \alpha & 0 \\ 0 & \alpha^{-1} \end{smallmatrix}\right ] \right\| \coloneqq \max\{\alpha,\alpha^{-1}\}. 
\]
Extend  the norm to a $K$-bi-invariant function $G\to [1,\infty)$ using the Cartan decomposition $G= K A  K$:
\[
\|k_1 a k_2\| \coloneqq \|a\|. 
\]
\end{definition}

\begin{remark}
One has $\log\|g\| = d(K,gK)$ for the standard $G$-invariant Riemannian metric on the Poincar\'e disk $G/K$. It thus follows from the triangle inequality that 
\begin{equation}\label{scale_inequality} 
\|gh\| \leq \|g\|\cdot \|h\| 
\end{equation}
for all $g,h\in G$. For the purposes of defining $\HC(G)$, this choice of norm is only one of several natural options; see \cite[Section 4.2]{Bernstein-Plancherel} and \cite[Section 2.1]{BK}.
\end{remark}

\begin{proposition}
\label{prop-almostL2}
If $t\ge 0$ is sufficiently large, then 
\[
\int_G \Xi_G(g)^2 (1+\log\|g\|)^{-t}\, \dd g < \infty.
\]
\end{proposition}

For a proof, see for instance \cite[Section 4.5]{Wallach1}.

\begin{definition}\label{HC_definition}
The \emph{Harish-Chandra Schwartz algebra} $\HC(G)$ is the space of all smooth, complex-valued functions $f$ on $G$ for which 
\begin{equation}\label{HC_seminorms}
\sup_{g\in G}\frac{| (X  f  Y)(g) |(1+\log\|g\|)^p}{\Xi_G(g)}   <\infty 
\end{equation}
for every $p\geq 0$ and every pair of invariant differential operators $X,Y\in \U(\germ g)$ (the enveloping algebra of the Lie algebra of $G$).
\end{definition}

The seminorms appearing in \eqref{HC_seminorms} make $\HC(G)$ into a nuclear Fr\'echet space.  Proposition~\ref{prop-almostL2} shows that the convolution of two Harish-Chandra functions is defined pointwise, and is a bounded function on $G$. A simple additional argument using \eqref{eq-spherical-fn} shows that the convolution product lies in $\HC(G)$, and that indeed $\HC(G)$ is  a Fr\'echet algebra under convolution. See \cite[Section 7.1]{Wallach1}, and compare also the proof of Lemma~\ref{lem-HC-convolution} below.

Our aim is to study the category of $\SF$-modules over ${\HC(G)}$, but to conclude this section we shall make some remarks concerning the     relationship between $\SF$-modules over $\HC(G)$, which we shall sometimes refer to as \emph{tempered} $\SF$-modules,  and $\SF$-modules over $\Sc(G)$. 

 \begin{lemma}
  The algebra $\Sc(G)$ embeds continuously as a dense subalgebra of $\HC(G)$, and $\HC(G)$ is     a left and right $\SF$-module over  $\Sc(G)$.  
\end{lemma}

\begin{proof}
For the first statement see for example   \cite[Theorem 7.1.1]{Wallach1}. 
It is easy to show directly that $\HC(G)$ is a smooth representation of moderate growth under left or right translation, so the second statement follows from Propositions~\ref{prop-moderate-growth} and \ref{prop-SF-equals-SF}.
\end{proof}

\begin{proposition} 
The restriction to $\Sc(G)$ of any $\SF$-module over $\HC(G)$ is an $\SF$-module over $\Sc(G)$.  
\end{proposition}

\begin{proof}
If $V$ is an $\SF$-module over $\HC(G)$, then 
\[
V  =  \HC(G)\otimes_{\HC(G)}V = \Sc(G)\otimes_{\Sc(G)} \HC(G)\otimes_{\HC(G)} V  = \Sc(G)\otimes_{\Sc(G)}V,
\]
as required.
\end{proof}

  In the reverse direction, we have the following result: 

\begin{proposition}
If $V$ is an $\SF$-module over $\Sc(G)$, and if the action of $\Sc(G)$ extends continuously to $\HC(G)$, then $V$ is an $\SF$-module over $\HC(G)$.  
\end{proposition}

\begin{proof}
The composition
\[
\Sc(G)\otimes _{\Sc(G)} V \longrightarrow \HC(G)\otimes _{\HC(G)} V \longrightarrow V 
\]
is an isomorphism, and so the left-hand map is a split injection.  It follows that its range is closed.  But  the range is dense, so the left-hand map is surjective too.  Hence  it is an isomorphism by the closed graph theorem, and the lemma follows.
\end{proof}

Finally, a comment on notation regarding tensor products  that we shall use from now on:

\begin{notation} 
\label{rem-balanced-tp-notation}
Assume that $V$ and $W$ are left  and right $\SF$-modules over $\Sc(G)$, respectively.  The balanced tensor product $W\otimes _{\Sc(G)} V$ does not depend on the particular details of the definition of $\Sc(G)$, since to form the balanced tensor product we need only form the quotient by the closed span of the relators $wf \otimes v - w \otimes fv$ with $f\in C_c^\infty (G)$.  
 With this in mind, and to streamline notation a little, we shall   write $W\otimes _G V$ for the balanced tensor product of $\Sc(G)$-modules.  We shall use the same notation in the $\HC(G)$-module case, and it will be important to note that in this case the $\Sc(G)$- and $\HC(G)$-module balanced tensor products are the same.
 \end{notation}
 
\section{Parabolic induction and restriction}
\label{sec-parabolic}

In this section we shall define parabolic induction and parabolic restriction of $\SF$-modules, in both the general and tempered contexts.

We continue with the notation established in the last section with regard to $G=\SL(2,\R)$ and its subgroups, except that throughout this section $N$ will denote either $\plus{N}$ or $\op{N}$, and $P$ will denote the corresponding parabolic subgroup. All of the results of this section are in fact valid for an arbitrary real reductive group $G$, where the notation $K$, $A$, etc. is given its customary meaning, as explained in \cite[Chapter 2]{Wallach1}, for example.

The homogeneous space $G/N$ is a nonsingular real algebraic variety, and hence a Nash manifold, so it possesses its own  space of Schwartz functions $\Sc(G/N)$.  We let $G$ act on $\Sc(G/N)$ in the standard way, by left translation:
\begin{equation}\label{eq-left-G-action}
(g{\cdot}h)(x) =  h(g^{-1}x) .
\end{equation}
We let the diagonal subgroup $L$ act by right translation, but shifted by the quasicharacter  $\delta$ of \eqref{delta_equation}:
\begin{equation}
\label{eq-right-L-action}
(h {\cdot} \ell) (x) = \delta(\ell)^{-1/2} h(x\ell^{-1}).
\end{equation}
Here $\delta=\pnm{\delta}$ if $N=\pnm{N}$. One explanation for the appearance of $\delta$ in the formula for the  action  is that with the $\delta$-factor   the action is unitary for the $L^2$-inner product associated to the $G$-invariant measure on $G/N$.

\begin{remark}
\label{rem-GmodN-space}
If $G= \SL(2,\R)$ and if say $N=\op{N}$, then the homogeneous space $G/N$ may be identified with the complement of the origin in  $\R^2$ via the map that sends the coset $gN$ to the second column  of the matrix $g$. Under this identification the left action of $G$ on the homogeneous space is through matrix multiplication, while the right action of $L$ is through scalar multiplication by the $(2,2)$-entry of $\ell\in L$.  The space  $\Sc(G/N)$ gets identified in this way with the usual space of Schwartz functions on $\R^2$  that  vanish to all orders at the origin. See \cite[Theorem 5.4.3]{AizGour}.    Similar observations apply to $N = \plus{N}$, of course.
\end{remark}

An application of Proposition~\ref{prop-SF-equals-SF} gives: 

\begin{lemma}
The $G$- and $L$-actions defined above give $\Sc(G/N)$ the structure of a left $\SF$-module over $\Sc(G)$ and a right $\SF$-module over $\Sc(L)$. \qed
\end{lemma}

\begin{remark}
Under the obvious identification of $L$ with $\R^\times$ via the $(1,1)$-matrix entry, the space $\Sc(L)$ gets identified with the space of usual Schwartz functions on $\R$ which vanish to all orders at $0$.
\end{remark}

\begin{definition}\label{Ind_definition}
The functor of \emph{parabolic induction}
\[
\Ind_P^G \colon \SFMod{\Sc(L)} \longrightarrow \SFMod{\Sc(G)}
\]
is the tensor product functor 
\[
\Ind_P^G \colon  V \longmapsto  \Sc(G/N) \otimes _{L} V  .
\]
\end{definition}

This is connected to the  more familiar definition of parabolic induction in the following  way: 

\begin{proposition}\label{Ind_ind_proposition}
The above functor   of parabolic induction  from $L$ to $G$ is naturally isomorphic to the  functor from $\SF$-representations of $L$ to $\SF$-representations of $G$ that associates to   $\pi\colon L \to \Aut (V)$   the representation 
\begin{equation*}
\bigl \{\, \phi \colon G\to V \,:\, \text{$\phi$ is smooth and  $\phi(g\ell n) =  \delta(\ell )^{-1/2 }\pi (\ell)^{-1} \phi(g)$} \, \bigr  \}   ,
\end{equation*}
on which the action of $G$   is by left translation.
\end{proposition}

\begin{proof}
Denote by  $\ind_P^G V$   the space  described in the statement of the lemma.
The map 
\[
\alpha\colon  \Sc(G/N)\otimes V \longrightarrow 
\ind_P^G  V
\]
that sends $h\otimes v \in \Sc(G/N)\otimes V$ to the function 
\[
g \longmapsto \int _L h (g\ell N ) \delta (\ell)^{1/2} \pi( \ell)v \, d\ell
\]
induces  a continuous, $\Sc(G)$-module homomorphism
\[
\Sc(G/N) \otimes _{L} V \longrightarrow \ind_P^G V.
\]
Fix    a smooth, compactly supported function   $\chi$ on $G/N$ with integral $1$ over each right $L$-orbit, and define 
\[
\sigma \colon  
\ind_P^G V
\longrightarrow 
 \Sc(G/N)\otimes V 
\]
by mapping $\phi \in\ind_P^G V $ to  the $V$-valued function 
\[
gN \mapsto \phi (gN)\chi(gN),
\]
which lies in $\Sc(G/N,V) \cong \Sc(G/N)\otimes V$.   In addition, define 
\[
\tau \colon  \Sc(G/N)\otimes V
\longrightarrow
 \Sc(G/N)\otimes \Sc(L) \otimes V 
 \]
 by means of the formula
 \[
 h \otimes v \longmapsto \bigl [ \strut(x,\ell) \mapsto \chi(x)\delta(\ell)^{1/2} h (x\ell) v \bigr ]
 \]
(regarding these formulas, compare the proof of Proposition~\ref{prop-SF-equals-SF}). 
We compute that  $\alpha\circ \sigma = \mathrm{id}$, and if
\[
\beta \colon \Sc(G/N)\otimes \Sc(L) \otimes V \longrightarrow \Sc(G/N)\otimes V
\]
is the balancing homomorphism 
\[
h\otimes f \otimes v \longmapsto hf\otimes v - h\otimes fv, 
\]
then 
 \[
\beta \circ \tau + \sigma \circ \alpha = \mathrm{id} \colon  \Sc(G/N)\otimes V \longrightarrow  \Sc(G/N)\otimes V .
\] 
The proof follows from this.
\end{proof}

Alongside parabolic induction we shall also study \emph{parabolic restriction}, which is defined in the context of $\SF$-modules as follows.  Begin with the right  homogeneous space $N\backslash G$. The real-algebraic structure on $N\backslash G$ yields a space of Schwartz functions $\Sc(N\backslash G)$, which becomes  a right $\SF$-module over $\Sc(G)$ under the $G$-action
\[
(h\cdot g) (y) = h(yg^{-1}) 
\]
 and a   left $\SF$-module over $\Sc(L)$under the shifted (unitary)  $L$-action 
 \[
( \ell\cdot  h) (y) = \delta(\ell)^{1/2} h(\ell^{-1}y). 
 \]

\begin{definition}
\label{def-parabolic-restriction}
The functor of \emph{parabolic restriction} 
\[
\Res_P^G : \SFMod{\Sc(G)} \longrightarrow \SFMod{\Sc(L)}
\]
is defined by $X \mapsto \Sc(N\backslash G)\otimes_{G} X$. \end{definition}

The definition will be justified by the reciprocity theorems to be proved in the coming sections.   
   For now, we   turn to parabolic induction in the tempered context.

\begin{definition} \label{Xi_GmodN_definition}
For $G=\SL(2,\R)$, and $N=\pnm{N}$ either the upper- or the lower-triangular unipotent subgroup, let 
\[
\Xi_{G/N}:G/N\longrightarrow  \R^+
\]
be the function
\[
\Xi_{G/N}(x) =  \delta(x)^{-1/2}
\]
where $\delta=\pnm{\delta}$ if $N=\pnm{N}$ (recall from \eqref{eq-delta-def2} that $\delta$ was extended from $L$ to $G$ as a function that is in particular right-$N$-invariant; so $\delta$ descends to the homogeneous space $G/N$).
\end{definition}

\begin{lemma}\label{Xi_GmodN_spherical_lemma}
For every $g\in G$ and $x\in G/N$ one has
\[
\frac{1}{\vol(K)} \int_{K} \Xi_{G/N}(gkx)\, \dd k = \Xi_G(g)\Xi_{G/N}(x). 
\]
\end{lemma}

\begin{proof}
Considered as functions of $x$, both the left and the right hand sides are left-$K$-invariant, and since $G=KAN$ we may assume that $x\in A$. Since $A$ normalizes $N$, the function $\delta:G\to \R^+$ satisfies 
\[
\delta(g)\delta(a)=\delta(ga) 
\]
for all $g\in G$ and $a\in A$, and thus the asserted equality follows from the definitions of $\Xi_G$ and $\Xi_{G/N}$.
\end{proof}

\begin{definition}\label{scale_GmodN_definition}
Let $\|\,\,\|:G/N\to [1,\infty)$ be the function 
\[
\left\| gN \right\|  =  \inf_{n\in N} \| gn\| . 
\]
More explicitly, if $g=kan$ in the Iwasawa decomposition, then $\|gN\| = \| a\|$.
\end{definition}

\begin{remark}
The inequality \eqref{scale_inequality} gives 
\begin{equation}\label{scale_GmodN_inequality} 
\| gx\| \leq \|g\| \cdot \|x\| 
\end{equation}
for all $g\in G$ and $x\in G/N$. \end{remark}

\begin{example}\label{Xi_GmodN_example}
Take $N=\plus{N}$, and identify $G/N$ with the complement of the origin in $\R^2$. Then, in polar coordinates, one has 
\[
\Xi_{G/N}(r,\theta) = r^{-1} \quad \text{and}\quad \|(r,\theta)\| = \max \{r,r^{-1}\}.
\]
\end{example}

\begin{definition}\label{HC_GmodN_definition}
The \emph{Harish-Chandra Schwartz space} $\HC(G/N)$ is the space of smooth, complex valued functions $h$ on $G/N$ for which 
\begin{equation}\label{HC_GmodN_seminorms}
\sup_{x\in G/N}\frac{| (X  h  Y)(x) |(1+\log\|x\|)^p}{\Xi_{G/N}(x)}   <\infty 
\end{equation}
for every $p\geq 0$ and every pair of invariant differential operators $X\in \U(\germ g)$ and $Y\in \U(\germ l)$.
\end{definition}

The seminorms appearing in \eqref{HC_GmodN_seminorms} make $\HC(G/N)$ into a nuclear Fr\'echet space, containing $\Sc(G/N)$ as a dense subspace. Compare \cite[Section 15.3]{Wallach2}.

\begin{example}
Continuing Example \ref{Xi_GmodN_example}, we see that $\HC(G/\plus{N})$ is the space of smooth functions on the complement of the origin in $\R^2$ all of whose derivatives satisfy estimates 
\[
\sup_{(r,\theta)} |h(r,\theta)|\cdot r(1+|\log r|)^p < \infty
\]
for all $p\geq 0$. Notice in particular that such functions need not vanish at the origin.
\end{example}

The groups $G$ and $L$ act on $\HC(G/N)$ by left and by right translation, as in \eqref{eq-left-G-action} and \eqref{eq-right-L-action}. These actions make $\HC(G/N)$ into an $\SF$-bimodule over $\Sc(G)$ and $\Sc(L)$, but in fact more is true:

\begin{lemma}
\label{lem-HC-convolution}
The left $\Sc(G)$ action and the right  $\Sc(L)$ action on $\HC(G/N)$ extend continuously to actions of $\HC(G)$ and $\HC(L)$, respectively.
\end{lemma}

\begin{proof}
The following argument is essentially the same as the proof that $\HC(G)$ is an algebra under convolution.
Fix $f\in \HC(G)$, $h\in \HC(G/N)$, and $p\geq 0$. To begin, we want to obtain an estimate of the form  
\[ 
|fh(x) | (1+\log\|x\|)^p \leq C\cdot \Xi_{G/N}(x),
\]
where $C$ depends on suitable seminorms of $f$ and $h$.

 Choose $t \ge 0$ satisfying the $L^2$-bound in Proposition~\ref{prop-almostL2}, and  then choose   $r,s\geq 0$ so that 
\[
(1+\log\|x\|)^p (1+\log\|g\|)^t \leq (1+\log\|g^{-1}x\|)^r (1+\log\|g\|)^s 
\]
for all $g\in G$ and $x\in G/N$; this is possible by virtue of \eqref{scale_GmodN_inequality}.  For $x\in G/N$ we make the estimate
\begin{multline}
\label{eq-estimate1}
\int_G |f(g)h(g^{-1}x)| \, dg (1+\log\|x\|)^p 
\\
\leq |f|_s |h|_r \int_G \Xi_G(g) \Xi_{G/N}(g^{-1}x) (1+\log\|g\|)^{-t} \, dg,
\end{multline}
where 
\[
|f|_s = \sup_{g\in G} \frac{|f(g)|(1+\log\|g\|)^s}{\Xi_G(g)} \qquad \text{and}\qquad |h|_r = \sup_{x\in G/N} \frac{ |h(x)|(1+\log\|x\|)^r}{\Xi_{G/N}(x)}. 
\]
The measure $dg$, the function $\Xi_G$ and  the norm $\|g\|$ all $K$-invariant, so we may rewrite the the right-hand side of \eqref{eq-estimate1}  as
\[ 
\int_G \Xi_G(g) (1+\log\|g\|)^{-t} \left( \frac{1}{\vol(K)} \int_K \Xi_{G/N}(g^{-1}kx)\, dk\right) \, \dd g.
\]
The integral inside the parentheses is equal to $\Xi_G(g)\Xi_{G/N}(x)$ by Lemma \ref{Xi_GmodN_spherical_lemma} and by the equality $\Xi_G(g^{-1})=\Xi_G(g)$. This gives us an estimate
\[ 
 |fh(x)| (1+\log\|x\|)^p \leq \left(\int_G \Xi_G(g)^2 (1+\log\|g\|)^{-t} \, dg\right)  |f|_s   |h|_r \cdot \Xi_{G/N}(x) ,
\]
as required.  Differentiating under the integral gives similar estimates for the left and right derivatives of $fh$, and these estimates show that $\HC(G/N)$ is a continuous left module over $\HC(G)$. 

A similar (simpler) argument shows that $\HC(G/N)$ is also a continuous right module over $\HC(L)$.
\end{proof}

\begin{remark}
One can define a satisfactory  function     $\Xi_{G/N}$, and using it a Harish-Chandra space $\HC(G/N)$, for a general real reductive group $G$ and  para\-bolic subgroup $P=L N$, as follows. Denote by $\Xi_L$ the  Harish-Chandra function  for the reductive group $L$, and by $\delta_L$ the homomorphism introduced in \eqref{delta_equation}. Extend both $\Xi_L$ and $\delta_L$ to smooth functions on $G$ that are  left $K$-invariant and right $N$-invariant, and then define  
$
\Xi_{G/N}  \colon G \longrightarrow  \R^+ 
$
by
\[
\Xi_{G/N}(x)  = \Xi_L (x)  \delta_L(x)^{-1/2} .
\] 
The proof of Lemma~\ref{lem-HC-convolution} above carries over to this context without change. 
\end{remark}

\begin{definition}
The functor of \emph{tempered parabolic induction}
\[
\Ind_P^G \colon \SFMod{\HC(L)} \longrightarrow \SFMod{\HC(G)}
\]
is defined by 
$
X \mapsto \HC(G/N) \otimes _{L} X
$.
\end{definition}

Actually, tempered parabolic induction is the same as ordinary parabolic induction restricted to the full  category of tempered $\SF$-modules, by virtue of the following result:

\begin{proposition}
\label{prop-Sc-ind-versus-HC}
 Let $X$ be an $\SF$-module over $\HC(L)$.  The natural morphism 
\[
 \Sc(G/N) \otimes _{L} X \longrightarrow  \HC(G/N) \otimes _{L} X 
\]
is an isomorphism of $\SF$-modules over $\Sc(G)$.
\end{proposition}

\begin{proof} 
It suffices to show that 
the natural morphism
\begin{equation}
\label{eq-Sc-ind-versus-HC}
\Sc(G/N)\otimes_{L} \HC(L)  \longrightarrow \HC(G/N)
\end{equation}
is an isomorphism of $\SF$-modules.
 The Iwasawa decomposition $G=KAN$ gives an  $L$-equivariant identification 
 \[
 G/N \cong K\times_{K\cap L} L,
 \]
 and corresponding isomorphisms
 \[
\Sc(K)\otimes_{ K\cap L } \Sc(L) \stackrel\cong \longrightarrow  \Sc(G/N) 
\]
and
\[
\Sc(K)\otimes_{K\cap L} \HC(L)  \stackrel\cong \longrightarrow  \HC(G/N),
 \]
both given explicitly by the formula
 \[
 f\otimes h \longmapsto \left[ k\ell N \mapsto \int_{K\cap L} f(km)h(m\ell) \delta(\ell)^{-1/2}\, \dd m \right].
 \]
 The morphism \eqref{eq-Sc-ind-versus-HC}  factors in this way: 
\[
 \xymatrix{
 \Sc(G/N) \otimes_{L} \HC(L) \ar[r] & \HC(G/N) \\ 
\Sc(K)\otimes_{ K\cap L} \Sc(L) \otimes_{L} \HC(L)  \ar[u]^-{\cong} \ar[r]_-{\cong} &  \Sc(K)\otimes_{ K\cap L} \HC(L) \ar[u]_-{\cong} 
 }
 \]
 and so it is an isomorphism. 
 \end{proof}
 
Finally, let us define parabolic restriction in the tempered context.  The Harish-Chandra space $\HC(N\backslash G)$ is defined as in Definition \ref{HC_GmodN_definition}, using the functions 
\[
\Xi_{N\backslash G}(Ny) \coloneqq \Xi_{G/N}(y^{-1}N) \qquad \text{and}\qquad \|Ny\|\coloneqq \|y^{-1}N\|. 
\]

\begin{definition}
\label{def-tempered-parabolic-restriction}
 The functor of 
\emph{tempered parabolic restriction} 
\[
\Res_P^G \colon \SFMod{\HC(G)} \longrightarrow \SFMod{\HC(L)}
\] 
is defined by 
$X \mapsto \HC(N\backslash G)\otimes _{G} X$.
\end{definition}

 Unlike the situation with parabolic induction, the functors of parabolic  {restriction} and tempered  {parabolic} restriction   \emph{differ} on the category of tempered $\SF$-modules over $G$, where they are both defined.  We shall briefly study the two  functors together in the next section and   show that they both satisfy Frobenius reciprocity.  Thereafter we shall focus on tempered parabolic restriction.  The second adjoint theorem that we are aiming for holds only in the tempered context.
 
\section{Frobenius reciprocity}
\label{sec-frobenius}

Throughout this section we shall denote by $N$ either one of the unipotent groups $\plus{N}$ or $\op{N}$, and we shall denote by $P$ the corresponding parabolic subgroup.  We shall prove that the parabolic restriction functor $\Res_P^G$ is left-adjoint to the parabolic  induction functor $\Ind_P^G$, in either the general or the tempered context.  The argument is essentially the same as one that we have given in a related context \cite{CH_cb}, and so we shall be brief.

We shall consider first the parabolic induction functor 
\[
\Ind_P^G \colon \SFMod{\Sc(L)} \longrightarrow \SFMod{\Sc(G)} ,
\]
and treat the tempered case afterwards. The adjunction isomorphism is induced from  an $\Sc(L)$-bimodule map
\begin{equation}
\label{eq-first-frob-counit2}
\Fr\colon   \Sc(N\backslash G) \otimes _{G} \Sc(G/N)    \longrightarrow \Sc(L) ,
\end{equation}
that corresponds  to restriction of functions of the closed subset $L \subseteq N\backslash G / N$ (we shall give the precise definition in a moment).  Given a morphism
\[
T \in \Hom_{\Sc(G)} \bigl ( X,\Sc(G/N)\otimes _{L} Y\bigr ) 
\]
we  form the composition
\begin{equation}
\label{eq-frob-iso-def}
 \Sc(N\backslash G)\otimes _{G}X \xrightarrow{\, \mathrm{id}\otimes T\, }   \Sc(N\backslash G)\otimes _{G}\Sc(G/N)\otimes _{L} Y \xrightarrow{\Fr\otimes \mathrm{id}}  Y ,
 \end{equation}
 and so obtain a   map
\begin{equation}
\label{eq-frob-iso-def2}
\Hom_{\Sc(G)} \bigl ( X,\Sc(G/N)\otimes _{L} Y\bigr ) 
\longrightarrow
 \Hom_{\Sc(L)} \bigl ( \Sc(N\backslash G)\otimes _{G}X, Y\bigr ) .
\end{equation}
 Compare \cite[Chapter IV]{MacLane} for this standard type of construction, and for the language used next:
 
\begin{theorem}[Frobenius reciprocity]
\label{thm-frobenius}
The map  \eqref{eq-frob-iso-def2}  is a bijection.  That is,  the functor
\[
\Res_{P}^G: \SFMod{\Sc(G)} \longrightarrow  \SFMod{\Sc(L)}
\]
is left adjoint to the functor
\[
\Ind^G_{P} :\SFMod{\Sc(L)} \longrightarrow \SFMod{\Sc(G)} 
\]
via the counit morphism \eqref{eq-first-frob-counit2}.
\end{theorem}

Here are the details concerning the construction of \eqref{eq-first-frob-counit2}.  The subgroup $LN \subseteq G$ is a closed Nash submanifold, and as a result restriction of Schwartz functions from $G$  to $LN$ gives a continuous linear map of Schwartz spaces
\begin{equation}
\label{eq-restr-to-closed}
  \Sc (G) \longrightarrow \Sc (LN) .
\end{equation}
This is left- and right-equivariant for the natural  actions of both $\Sc(L)$ and $\Sc(N)$.  The basic idea behind \eqref{eq-first-frob-counit2} is to form  the tensor product of \eqref{eq-restr-to-closed} on both sides by the trivial $\Sc(N)$-module $\C$: 
\begin{equation}
\label{eq-first-frob-counit1}
 \C  \otimes _{N} \Sc(G) \otimes _{N} \C    \longrightarrow \C \otimes _{N}  \Sc (LN) \otimes _{N} \C  .
\end{equation}
 The left-hand side of \eqref{eq-first-frob-counit1} identifies with $\Sc(N\backslash G) \otimes _{G} \Sc(G/N)$, while the right-hand side identifies with $\Sc(L)$, so we obtain an $\Sc(L)$-bimodule map \eqref{eq-first-frob-counit2} as required.    But to obtain the correct left and right actions of $L$ a little additional  care is required.

\begin{definition} Let  $\alpha\colon A \to \C^{\times}$ be a continuous homomorphism.
\begin{enumerate}[\rm (a)]
\item Denote by ${}_\alpha\C $  the one-dimensional space $\C$ viewed as a left    $\Sc(L)$-module via $\alpha$  (so that $\ell \cdot \lambda  =\alpha (\ell)  \lambda $)   and as a trivial  right $\Sc(L)$-module.  
\item Similarly, denote by $\C_\alpha$ the one-dimensional space $\C$ viewed as a right    $\Sc(L)$-module via $\alpha$ (so that $\lambda \cdot \ell = \lambda \alpha (\ell)$)  and as a trivial  left $\Sc(L)$-module. 
\end{enumerate}
We give  both ${}_\alpha\C $ and $\C_\alpha$ the trivial $N$-module structure.
\end{definition}

Now, in place of \eqref{eq-first-frob-counit1} we form the tensor product morphism
\begin{equation}
\label{eq-first-frob-counit1-1}
 {{}_{\delta^{-1/2}}\C}  \otimes _{N} \Sc(G) \otimes _{N} \C_{\delta^{1/2}}    \longrightarrow  {{}_{\delta^{-1/2}}\C} \otimes _{N}  \Sc (LN) \otimes _{N} \C_{\delta^{1/2}}  .
 \end{equation}
with the indicated diagonal actions of $L$.   The left-hand sides of  \eqref{eq-first-frob-counit1-1} and \eqref{eq-first-frob-counit2} are identified as $\Sc(L)$-bimodules via the diagram 
\begin{multline*}
     {{}_{\delta^{-1/2}}\C} \otimes _{N} \Sc(G) \otimes _{N}  \C_{\delta^{1/2}}
\stackrel \cong \longleftarrow  {{}_{\delta^{-1/2}}\C}  \otimes _{N} \Sc(G)\otimes_G\Sc(G) \otimes _{N}  \C_{\delta^{1/2}}  \\
  \stackrel \cong  \longrightarrow 
     \Sc(N\backslash G) \otimes _{G} \Sc(G/N) 
 \end{multline*}
in which the leftwards isomorphism is induced from the multiplication operation on $\Sc(G)$, while the rightwards morphism is the tensor product of the isomorphisms induced from integration over the right and left cosets of $N$:
\begin{equation}
\label{eq-int-over-cosets1}
   {{}_{\delta^{-1/2}}\C} \otimes _{N} \Sc(G)  \longrightarrow \Sc(N\backslash G) 
\quad
\text{and}\quad  \Sc(G)\otimes_{N} \C_{\delta^{1/2}} \longrightarrow \Sc(G/N)  .
 \end{equation}
 For instance the left-hand map in \eqref{eq-int-over-cosets1} is 
 \[
 \lambda \otimes f \longmapsto \left [ g \mapsto \lambda \int_N f(ng)\, \dd n \right ].
 \]
 As for the right-hand sides of  \eqref{eq-first-frob-counit1} and \eqref{eq-first-frob-counit2}, they are identified as $\Sc(L)$-bimodules by the isomorphism 
\[
 {}_{\delta^{-1/2}} \C\otimes_{N} 
 \Sc(LN) \otimes _{N} \C_{\delta^{ 1/2}} \longrightarrow \Sc(L)
 \]
given by the formula 
\[
\lambda  \otimes f \otimes \mu \longmapsto \left [ \ell \mapsto  \lambda \cdot \mu \cdot \delta (\ell)^{1/2} \int_N f(\ell n)\, \dd n\right ] .
\]

Turning to the  proof to Frobenius reciprocity, we shall begin by describing  the   restriction morphism  \eqref{eq-first-frob-counit2} in a different way.
Form the standard $L^2$-inner product 
\[
  \langle h_1, h_2\rangle_{L^2(G/N)}   = \int _{G/N} \overline{h_1 (x)} h_2 (x) \,\dd x
\]
 of complex-valued functions on $G/N$ (as we noted earlier, the right action of $L$ is unitary for this inner product). Given $h_1,h_2\in \Sc(G/N)$, let us now define 
\[
\llangle h_1,h_2\rrangle \in \Sc(L)
\]
by means of the formula
\[
\llangle h_1,h_2\rrangle \colon \ell \longmapsto   \langle h_1, h_2\cdot \ell^{-1}\rangle_{L^2 (G/N)}  .
 \]
As explained in \cite{Clare_pi,CCH1} this pairing enjoys a number of properties.  For instance  if $f_1,f_2\in \Sc(L)$, then
\[
\llangle h_1 f_1,h_2f_2\rrangle  = f_1{}^* \llangle h_1,h_2\rrangle f_2 ,
\]
where the multiplication is as usual  convolution, and  where 
\[
f^*(\ell) = \overline {f(\ell^{-1})}.
\]
  In addition, if $f\in \Sc(G)$, then 
\[
\llangle f h_1,h_2  \rrangle  = \llangle h_1, f^*  h_2  \rrangle  .
\]
The restriction morphism is expressible very simply in terms of $\llangle\,\,,\,\rrangle$ as follows:

\begin{lemma}
\label{lem-frob-counit-fmla}
If $h_1 \in \Sc(N\backslash G) $ and $h_2 \in \Sc(G/N)$, then
\[
\pushQED{\qed} 
\Fr \colon h_1 \otimes h_2 \longmapsto  \llangle h_1{}^*, h_2 \rrangle.
\qedhere
\popQED
\]
\end{lemma}

We are now ready to proceed to the proof of Frobenius reciprocity.  The formula 
\begin{equation*}
\label{eq-cpt-ops2}
h_1\otimes h_2 \longmapsto \bigl [  h \mapsto h_1 \llangle h_2{}^*, h\rrangle \bigr ] 
\end{equation*}
defines an $\Sc(G)$-bimodule morphism
\begin{equation}
\label{eq-cpt-ops1}
\Sc( G/N) \otimes _{L} \Sc (N\backslash G) \longrightarrow \End_{\Sc(L)} (\Sc(G/N)) .
\end{equation}
  Similarly the formula 
\begin{equation*}
\label{eq-cpt-ops3}
h_1\otimes h_2 \longmapsto \bigl [  h \mapsto \llangle h^*, h_1\rrangle h_2 \bigr ] 
\end{equation*}
defines an $\Sc(G)$-bimodule morphism
\begin{equation}
\label{eq-cpt-ops4}
\Sc( G/N) \otimes _{L} \Sc (N\backslash G) \longrightarrow \End_{\Sc(L)} (\Sc(N\backslash G)) .
\end{equation}

\begin{lemma} 
\label{lem-cpt-action}
The left action of $\Sc(G)$ on $\Sc(G/N)$ and the right action of  $\Sc(G)$  on $\Sc(N\backslash G)$   both factor through a single  $\Sc(G)$-bimodule morphism
\[
\Sc(G) \longrightarrow \Sc( G/N) \otimes _{L} \Sc (N\backslash G) .
\]
\end{lemma}

Taking this for granted for a moment, the proof of reciprocity is straightforward:
\begin{proof}[Proof of Theorem~\ref{thm-frobenius}]
The  morphism from the statement of Lemma~\ref{lem-cpt-action} gives rise to a linear transformation 
\[
 \Hom_{\Sc(L)} \bigl ( \Sc(N\backslash G)\otimes _{G}X, Y\bigr ) 
\longrightarrow
\Hom_{\Sc(G)} \bigl ( X,\Sc(G/N)\otimes _{L} Y\bigr )  .
\]
To prove that this is inverse to the map in the statement of Theorem~\ref{thm-frobenius} it suffices to show that the composition
\begin{multline*}
\Sc(G) \otimes _G \Sc (G/N) \longrightarrow 
\Sc(G/N) \otimes _L \Sc (N \backslash G) \otimes _G \Sc(G/N)
\\
 \longrightarrow
\Sc(G/N)\otimes _L \Sc(L) \longrightarrow 
\Sc(G/N)
\end{multline*}
is the multiplication map, as is  
\begin{multline*}
\Sc (N\backslash G )  \otimes _G  \Sc(G) \longrightarrow 
\Sc (N\backslash G )  \otimes _G \Sc(G/N) \otimes _L \Sc (N \backslash G) 
\\
 \longrightarrow
\Sc(L) \otimes _L \Sc (N \backslash G) \longrightarrow 
\Sc(N\backslash G ) .
\end{multline*}
Compare \cite[Chapter IV]{MacLane}.  These facts follow immediately from Lemma~\ref{lem-frob-counit-fmla} together with \eqref{eq-cpt-ops1} and \eqref{eq-cpt-ops4}.
\end{proof}

\begin{proof}[Proof of Lemma~\ref{lem-cpt-action}]
The tensor product $\Sc(G/N) \otimes _L \Sc(N\backslash G)$  may be identified as an $\SF$-bimodule over $\Sc(G)$ with the space  of smooth functions 
\[
k\colon G/N \times N \backslash G  \longrightarrow \C 
\]
such that 
\begin{enumerate}[\rm (i)]
\item  $k(x\ell, y) = \delta (\ell)^{-1} k(x,\ell y)$, and 
\item   the function 
\[
(g_1,g_2) \longmapsto \int _L \delta (\ell) k(x\ell, \ell^{-1} y) \, \dd \ell 
\]
is a Schwartz function on $G/N \times _L N\backslash G$.
\end{enumerate}
The isomorphism is 
\[
h_1 \otimes h _2 
\longmapsto 
\left [ 
(g_1, g_2 ) \mapsto  \int _L  h_1 (g_1\ell) \delta(\ell)  h_2 (\ell^{-1} g_2)  \, \dd \ell 
\right ]
\]
Using this description of $\Sc(G/N)\otimes _L \Sc(N\backslash G)$ we may define a morphism
\[
\Sc(G) \longrightarrow \Sc(G/N) \otimes _L \Sc(N\backslash G)
\]
by associating to $f\in \Sc(G)$ the smooth function 
\[
k_f (g_1,g_2)  =   \int _{N}  f(g_1ng_2^{-1})\, dn 
\]
satisfying (i) and (ii) above. Compare \cite{CCH1}.
\end{proof}

\begin{remark}
There is a more familiar adjunction isomorphism in the form
 \begin{equation}
 \label{rem-hom-tensor1}
 \Hom_{\Sc(L)} \bigl ( \Sc(N\backslash G)\otimes _{G}X, Y\bigr )
\stackrel \cong \longrightarrow
\Hom_{\Sc(G)} \bigl  ( X,\Hom_{\Sc(L)} \left ( \Sc(N\backslash G),  Y\right ) \bigr ) .
\end{equation}
It is related to the Frobenius reciprocity theorem proved in this section, as follows. The left $\Sc(G)$-module  $\Hom_{\Sc(L)} \left ( \Sc(N\backslash G),  Y\right ) $ is not an $\SF$-module (it is not a Fr\'echet space) but this problem can be addressed by substituting  the (completed, projective) tensor product
 \begin{equation}
 \label{rem-hom-tensor2}
 \Sc(G)\otimes_G  \Hom_{\Sc(L)} \left ( \Sc(N\backslash G),  Y\right ).
\end{equation}
 This does not affect the isomorphism \eqref{rem-hom-tensor1} as long as $X$ is an $\SF$-module.
The formula 
\[
h \otimes y \longmapsto \Bigl [ h_1 \mapsto  y \Fr (h_1 \otimes h) \Bigr ]
\]
defines a continuous map of $\Sc(G)$-modules
 \[
  \Sc(G/N) \otimes _{L} Y  \longrightarrow  \Hom_{\Sc(L)} ( \Sc(N\backslash G),  Y)  .
\]
This becomes an isomorphism after tensoring with $\Sc(G)$. Its inverse is the composition 
 \begin{multline*}
   \Sc(G)\otimes_G  \Hom_{\Sc(L)} \left ( \Sc(N\backslash G),  Y\right )
  \\
   \longrightarrow 
   \Sc(G/N) \otimes _{L} \Sc(N\backslash  G) \otimes _G \Hom_{\Sc(L)} \left ( \Sc(N\backslash G),  Y\right )
\\
  \longrightarrow
      \Sc(G/N) \otimes Y ,
 \end{multline*}
 that combines the unit transformation of Lemma~\ref{lem-cpt-action} with the obvious evaluation map.  In short, Lemma~\ref{lem-cpt-action} converts the standard Hom-Tensor adjunction isomorphism into the Frobenius theorem of this section.
\end{remark}

\begin{remark} 
See  \cite{CCH2} for a related adjunction isomorphism in the context of operator modules rather than $\SF$-modules.
\end{remark}

Finally, let us formulate and prove the Frobenius reciprocity theorem in the context of tempered $\SF$-modules.  The counit transformation \eqref{eq-first-frob-counit2} yields a bimodule map
\begin{equation}
\label{eq-first-frob-counit-tempered}
\Fr\colon   \HC(N\backslash G) \otimes _{G} \HC(G/N)    \longrightarrow \HC(L) 
\end{equation}
simply by tensoring \eqref{eq-first-frob-counit2} on the left and right by $\HC(L)$.

\begin{theorem}[Frobenius reciprocity for tempered modules]
\label{thm-tempered-frobenius}
The functor
\[
\Res_{P}^G: \SFMod{\HC(G)} \longrightarrow  \SFMod{\HC(L)}
\]
is left adjoint to the functor
\[
\Ind^G_{P} :\SFMod{\HC(L)} \longrightarrow \SFMod{\HC(G)}. 
\]
via the counit transformation \eqref{eq-first-frob-counit-tempered}.
\end{theorem}

\begin{proof}
Consider the diagram 
\[
\xymatrix{
\Sc(G) \ar[r] \ar[d] & \Sc(G/N)\otimes _L \Sc(N\backslash G) \ar[d] \\
\HC(G) \ar@{..>}[r] & \HC(G/N)\otimes _L \HC(N\backslash G)
}
\]
in which the  vertical maps are induced from  the     inclusions of $\Sc(G)$ into $\HC(G)$ and of $\Sc(L)$ into $\HC(L)$, and the top horizontal map is the unit transformation for Frobenius reciprocity from Lemma~\ref{lem-cpt-action}.    There is a continuous extension along the bottom dotted arrow, and the argument used to prove Theorem~\ref{thm-frobenius} can now be repeated. 
 \end{proof}

\section{The second adjoint theorem}
\label{sec-bernstein}

In this section we shall formulate our second adjoint theorem.    We shall      define a $\HC(L)$-bimodule morphism
\begin{equation}
\label{eq-tempered-bernstein-unit}
\Extn \colon \HC(L) \longrightarrow \HC(\op{N}\backslash G) \otimes _{G} \HC(G/\plus{N}) ,
\end{equation}
that corresponds  to the inclusion
of $L$ as an open subset of $\op{N}\backslash G/\plus{N}$.    It will serve as the unit for the  second adjunction isomorphism.

We shall begin by working with $\Sc(G)$ rather than $\HC(G)$, and produce an $\Sc(L)$-bimodule morphism
\begin{equation}
\label{eq-bernstein-unit1}
\Extn \colon \Sc(L) \longrightarrow \Sc(\op{N}\backslash G) \otimes _{G} \Sc(G/\plus{N}).
\end{equation}
 The Cartesian product $\op{N}L \plus{N}$ is a Nash-open set in $G$ (its complement is the closed subset $L\plus{N}\subseteq G$) and there is an associated  a continuous linear map 
\begin{equation}
\label{eq-extn-by-zero}
 \Sc(\op{N}L \plus{N}) \longrightarrow \Sc(G)
\end{equation}
that extends functions by zero from $\op{N}L \plus{N}$ to $G$.  Compare \cite[Proposition 4.3.1]{AizGour}.  
Form the tensor product  morphism
\begin{equation}
\label{eq-extn-by-zero2}
{{}_{\op{\delta}^{-1/2}}\C}  \otimes _{\op{N}}  \Sc(\op{N}L \plus{N}) \otimes _{\plus{N}}  \C_{\plus{\delta}^{1/2}} 
\longrightarrow 
{{}_{\op{\delta}^{-1/2}}\C}  \otimes _{\op{N}}  \Sc(G) \otimes _{\plus{N}}   \C_{\plus{\delta}^{1/2}} .
\end{equation}
The right-hand side in \eqref{eq-extn-by-zero2} identifies with 
$\Sc(\op{N}\backslash G) \otimes _G \Sc(G/N)$ as an $\SF$-bimodule over $\Sc(L)$ via the tensor product (over $G$) of  the bimodule isomorphisms
\begin{equation}
\label{eq-integration-isos}
{{}_{\op{\delta}^{-1/2}}\C} \otimes _{\op{N}} \Sc(G)  \stackrel{\cong}\longrightarrow  \Sc (\op{N}\backslash G) 
\quad \text{and} \quad  
\Sc(G) \otimes _{\plus{N}}  \C_{\plus{\delta}^{1/2}}  \stackrel{\cong}\longrightarrow  \Sc (G/\plus{N})  
\end{equation}
given by integration.  The left-hand side in \eqref{eq-extn-by-zero2}  identifies with $\Sc(L)$ as an $\SF$-bimodule  along the bimodule isomorphism
\[
{{}_{\op{\delta}^{-1/2}}\C}  \otimes _{\op{N}}  \Sc(\op{N}L \plus{N}) \otimes _{\plus{N}}  \C_{\plus{\delta}^{1/2}}  
\longrightarrow \Sc(L)
\]
given by the formula 
\begin{equation}
\label{eq-opNLplusN}
\lambda \otimes f \otimes \mu \longmapsto \left [   \ell \mapsto \lambda \cdot \mu\cdot  \plus{\delta}(\ell)^{1/2}  \int _{\op{N}}\int_{\plus{N}} f (\op{n} \ell \plus{n} ) \, \dd \op{n} \dd \plus{n}  \right ]
\end{equation}
We obtain an ``extension by zero'' morphism of bimodules \eqref{eq-bernstein-unit1} 
as required.  And then we obtain a tempered version \eqref{eq-tempered-bernstein-unit}  by tensoring  \eqref{eq-bernstein-unit1} on both sides  by $\HC(L)$ and invoking Proposition~\ref{prop-Sc-ind-versus-HC}.  With this, we can formulate our second adjoint theorem:
 
\begin{theorem}
[Second Adjoint Theorem]
\label{thm-second-adjoint}
The functor
\[
\Ind_{\plus{P}}^G: \SFMod{\HC(L)} \longrightarrow  \SFMod{\HC(G)}
\]
is left adjoint to the functor
\[
\Res^G_{\op{P}} :\SFMod{\HC(G)} \longrightarrow \SFMod{\HC(L)}
\]
via the unit morphism \eqref{eq-tempered-bernstein-unit}.
\end{theorem}

\begin{remark}
The theorem can be formulated in the context of $\SF$-modules over $\Sc(G)$ and $\Sc(L)$, using \eqref{eq-bernstein-unit1}, but in this context it is \emph{false}. 
\end{remark}

As one might expect in view of the previous remark, the proof of the second adjoint theorem is much more involved than the proof of Frobenius reciprocity, and we defer it to the next section (with some details  deferred further to an appendix).  In the remainder of this section we shall prepare for the proof by presenting a helpful explicit formula for the unit morphism.

\begin{lemma}
\label{lem-tech-unit-fmla}
Let  $f\in \Sc(L)$.  Let $\plus{v}$ and $\op{v}$  be   smooth, compactly supported functions on $\plus{N}$ and  $\op{N}$, respectively, both with total integral $1$,
and  define  
$
f_0\in \Sc( \op{N} L  \plus{N}) $ by   
\begin{equation}
\label{eq-f_0-def}
f_0 \colon \op{n} \cdot \ell \cdot \plus{n} \longmapsto   \op{v}(\op{n}) \cdot f(\ell) \plus{\delta}(\ell)^{-1/2}\cdot  \plus{v}(\plus{n}),
\end{equation}
and then  extend $f_0 $ by zero to obtain a Schwartz function on $G$.  
Let $u$ be a smooth compactly supported function on $G$ with total integral $1$, and define  
\[
k_0 \in \Sc(G) \otimes \Sc(G ) \cong  \Sc(G\times G )  
\] by 
\begin{equation}
\label{eq-k_0-def}
k_0
\colon (g_1 , g_2) \longmapsto   u(g_1 ) f_0 (g_1 g_2) 
\end{equation}
The image of $f\in \Sc(L)$ under the unit morphism \eqref{eq-bernstein-unit1} is equal to the image of  $k_0$ 
under the convolution and integration maps 
\[
 \Sc(G) \otimes \Sc(G )\longrightarrow  \Sc(G)\otimes _G \Sc(G )  
 \longrightarrow 
  \Sc(\op{N}\backslash G)\otimes _G \Sc(G/ \plus{N})  .
  \]
\end{lemma}

  \begin{proof} 
Consider the diagram 
 \[
 \xymatrix{
 \C \otimes_{\op{N}}  \Sc(\op{N}L\plus{N}) \otimes_{\plus{N}} \C \ar[dd] \ar[r] &  \C\otimes_{\op{N}}  \Sc(G) \otimes_{\plus{N}}  \C 
 \\
 &  \C\otimes_{\op{N}}  \Sc(G)\otimes _G \Sc(G) \otimes_{\plus{N}}  \C \ar[u]_{\cong}\ar[d]^{\cong} 
 \\
 \Sc(L) \ar[r]_- {\Extn} & \Sc(\op{N}\backslash G) \otimes _G \Sc(G/\plus{N})
 }
 \]
in which the top horizontal map is induced from \eqref{eq-extn-by-zero}, the downward vertical maps are given by integration over $N_\pm$ as in \eqref{eq-opNLplusN} and \eqref{eq-integration-isos},  and the upward vertical map is induced from the product on $\Sc(G)$.  (We have suppressed the decorations describing the left and right $L$-actions, since the $L$-actions are not relevant for the proof.)
The element $1\otimes f_0 \otimes 1$   in the top left 
 maps to $f$  in the bottom left of the diagram.  
 It maps to the same element in the top right as  does $k_0$.  This means that $k_0$  maps to the image of $f$ under the unit map $\Extn$, as required. 
\end{proof}

\section{Proof of the Second Adjoint Theorem}
\label{sec-proof} 

In this section we shall formulate a theorem concerning intertwining integrals and wave packets,  and use it to prove our second adjoint theorem.

\begin{lemma}
\label{lem-J-integral-convergence}
  If $k\in \HC(G/ \op{N})$, then for every $g\in G$ the integral 
\[
(\plus{J}k)(g) =  \int_{ \plus{N}} k(gn)\, \dd n 
\]
converges absolutely and defines a smooth function on $G$.  The resulting   linear map
\[
\plus{J}:\HC(G/\op{N})\longrightarrow  C^\infty(G/\plus{N})  
\]
is continuous. Similarly, the integral
\[
(\op{J}h)(g) = \int_{\op{N}} h(g{n})\, \dd {n}
\]
defines a continuous linear map 
\[
\op{J} :\HC(G/\plus{N}) \longrightarrow C^\infty(G/\op{N})  .
\]
\end{lemma}

These are the standard intertwining integrals considered throughout the representation theory of reductive groups, although it is more common to consider them as defined on individual parabolically induced representations rather than on the inducing bimodule as we are doing here.

\begin{remark} 
The Fr\'echet spaces $C^\infty (G/\plus{N})$ and $C^\infty (G/\op{N})$  carry continuous and indeed smooth left actions of $G$, but they are not  $\SF$-modules over $\Sc(G)$, let alone $\HC(G)$.  
\end{remark}

The proof of the lemma for general parabolic subgroups of general reductive groups is a bit involved, but for $G= \SL(2,
\R)$ it is straightforward: 

\begin{proof}[Proof of Lemma~\ref{lem-J-integral-convergence}]
We'll consider the first integral (the second is handled in the same way).  In view of the Iwasawa decomposition $G = KA\plus{N}$ it suffices to consider $g$ of the form $g=ka$,  and in order to estimate the integrand we need to solve for $a_1$ in 
\[
ka\plus{n}  = k_1 a_1 \op{n} .
\]
Multiplying the matrices on both sides of this equation on the left  by their transposes (which eliminates $k$ and $k_1$) we find that if 
\[
a = \begin{bmatrix} e^t & 0 \\ 0 & e^{-t} \end{bmatrix} 
\quad \text{and}\quad 
\plus{n} = \begin{bmatrix}1& x \\ 0 & 1 \end{bmatrix} 
\]
then
\begin{equation}
\label{eq-iwasawa-matrix-calc}
a_1^2 =  \begin{bmatrix} (e^{-2t} + x^2 e^{2t})^{-1/2}  & 0 \\ 0 & (e^{-2t} + x^2 e^{2t})^{1/2}\end{bmatrix}  .
\end{equation}
So according to the definition of the Harish-Chandra space $\HC(G/\op{N})$,  
\begin{equation}
\label{eq-J-estimate}
 | h(g{n}) |  \le \text{constant}_{h,p} \cdot  (e^{-2t} + x^2 e^{2t})^{-1/2}  \cdot (1 + \tfrac{1}{2}\log (e^{-2t} + x^2 e^{2t}) )^{-p}
\end{equation}
for all large $x$ (large enough that the logarithm is positive).  If $p>1$, then the bound is an integrable function of $x$, as required.  Smoothness of $\op{J}h$  follows by differentiating under the integral, and continuity follows from pointwise continuity of $\op{J}h$ as a function of $h$ and the closed graph theorem.
\end{proof}
 
Here is the more elaborate result about the intertwining integrals and wave packets that we shall require.  It fits into Harish-Chandra's general theory of the Plancherel formula, but we shall give a direct proof for $G=\SL(2,\R)$ in the final section of the paper.

\begin{theorem}\label{I_theorem}
There are continuous bimodule maps 
\[ 
\pnm{I} : \HC(G/\pnm N) \to \HC(G/ \mnp N) 
\]
such that: 
\begin{enumerate}[\rm(a)]
\item The map $\pnm I$ is a right-inverse to the standard intertwiner $\pnm J$:
\[ \pnm J \circ \pnm I = \id_{\HC(G/\pnm N)}.\]
\item The maps $\plus I$ and $\op I$ are $L^2$-adjoints of one another: 
\[
\langle k, \plus{I} h \rangle_{L^2(G/\op{N})} = \langle \op{I} k , h \rangle_{L^2 (G/\plus{N})}
\]
for all $h\in \HC(G/\plus N)$ and $k\in \HC(G/\op N)$.
\item The range of the  map 
\[
B:\HC(G/\plus{N})\otimes_{L} \HC(\op{N}\backslash G) \longrightarrow  C^\infty (G)
\]
defined by the formula
\[
B \colon h\otimes k \longmapsto \bigl [ g \mapsto   \langle k^*, g^{-1}  \plus{I}h  \rangle_{L^2 (G/\op{N})} \bigr ] 
\]
  lies in the vector subspace $\HC(G)\subseteq C^\infty (G)$, and $B$ is continuous as a map into $\HC(G)$.\hfil\qed
  \end{enumerate}
  \end{theorem}

\begin{remark}
The formula in part (c), which expresses the counit of the second adjunction in terms of the inverse of the standard intertwining operator, is an analogue in our tempered, real context  of the formula \cite[Corollary 7.8]{BezKaz} of Bezrukavnikov and Kazhdan.
\end{remark}

Taking Theorem~\ref{I_theorem} for granted, we  shall  now prove  Theorem~\ref{thm-second-adjoint} by calculating that the morphisms  $B$ and $\Extn$ are the counit and unit of an adjunction.  That is, we shall show that the compositions
\begin{multline}
\label{eq-first-triangle-relation}
\HC(G/\plus{N})\otimes_{L} \HC(L)
	\xrightarrow{1\otimes \Extn} 
			 \HC(G/\plus{N})\otimes_{L} \HC(\op{N}\backslash G)\otimes_{G} \HC(G/\plus{N})
 \\
\xymatrix{ \ar[r]^-{B\otimes 1}  & \HC(G)\otimes_{G} \HC(G/\plus{N}) \ar[r]^-{\conv_G} &  \HC(G/\plus{N})
}
\end{multline}
and
\begin{multline}
\label{eq-second-triangle-relation}
\HC(L)\otimes_{L} \HC(\op{N}\backslash G) \xrightarrow{\Extn\otimes 1}  \HC(\op{N}\backslash G) \otimes_{G} \HC(G/\plus{N})\otimes_{L} \HC(\op{N}\backslash G)
\\
\xymatrix{
 \ar[r]^-{1\otimes B}  & \HC(\op{N}\backslash G)\otimes_{G} \HC(G) \ar[r]^-{\conv_G} &  \HC(\op{N}\backslash G) 
 }
\end{multline}
are the canonical isomorphisms.  Compare  \cite[Chapter IV]{MacLane}  again.  

Throughout the following calculations we shall regard functions on homogenous spaces of $G$ as functions on $G$ itself that are constant on the appropriate left or right cosets.

\begin{lemma}
\label{lem-tech-second-adjoint}
Suppose that  $h \in \HC(G/\plus{N})$ and let 
\[
h_1 = \plus{I}h \in \HC(G/\op{N}).
\]
 Suppose  that an element 
$
k\in \HC(\op{N}\backslash G )\otimes _G \HC(G/\plus{N})
$
 is obtained by integrating a function
\[
k_1\in \Sc(G \times G/\plus{N})
\]
over left $\op{N}$-cosets in the first variable \textup{(}and then projecting to the balanced tensor product\textup{)}.  The image of the element 
\[
h\otimes k \in \HC(G/\plus{N})\otimes _L \HC(\op{N}\backslash G) \otimes _G \HC(G/\plus{N})
\]
under the map \eqref{eq-first-triangle-relation} is the right $\plus{N}$-invariant function 
\begin{equation}
\label{eq-first-triangle-for-bernstein}
g \longmapsto  \int_G \int_G k_1(\gamma_1^{-1}, \gamma_1\gamma_2^{-1} g) h_1 (\gamma_2 )\, \dd \gamma_1 \dd \gamma_2 .
\end{equation}
\end{lemma}

\begin{proof}
If we apply first the map $B$, then we obtain from $h\otimes k$ the element in $\HC(G)\otimes_G \HC(G/\plus{N})$ represented by the function 
\[
(g_1, g) \longmapsto 
	\int _{G/\op{N}} \left ( \int _{\op{N}} k_1(nx^{-1}, g) \, \dd n \right )h_1 (g_1x)\, \dd x .
\]
Since $h_1$ is right $\op{N}$-invariant, we can combine the integrals and write the above function as 
\[
(g_1, g) \longmapsto 
	\int _{G}  k_1(g_2{}^{-1}, g) h_1 (g_1g_2)\, \dd g_2 .
\]
Applying to this the convolution map 
\[
\HC(G) \otimes \HC(G/\plus{N}) \longrightarrow \HC (G/\plus{N})
\]
we obtain the element of $\HC(G/\plus{N})$ represented by the function 
\[
g \longmapsto \int _G \int_G k_1(g_2{}^{-1}, g_1{}^{-1}g) h_1 (g_1g_2)\, \dd g_1\dd g_2
\]
Now make the substitutions $\gamma_1 = g_2$ and $\gamma_2 = g_1g_2$ to obtain \eqref{eq-first-triangle-for-bernstein}.
\end{proof}

We are going to apply   Lemma~\ref{lem-tech-second-adjoint} in the case where the element  $k$ has  the form 
\[
k = \Extn(f)
\]
 for some $f\in \Sc(L)$.  In this case, according to  Lemma~\ref{lem-tech-unit-fmla}, and using the notation introduced there,  the function $k_1$ in the statement of the lemma  above may be defined by 
 \[
 k_1(g_1,g_2) = \int _{\plus{N}} u(g_1) f_0(g_1g_2n)\, \dd n.
 \]
From the formula 
 \[
  k_1(\gamma_1^{-1}, \gamma_1\gamma_2^{-1} g) 
  =
  \int_{\plus{N}} u(\gamma_1^{-1} ) f_0 (\gamma_2^{-1}gn)\, \dd n
  \]
we find that after integrating over $\gamma_1\in G $  in \eqref{eq-first-triangle-for-bernstein} we obtain the function
\begin{equation*}
g \longmapsto   \int_G \int _{\plus{N}} f_0(\gamma_2^{-1} g n) h_1 (\gamma_2 )\,  \dd \gamma_2 \dd n  ,
\end{equation*}
 which in view of  the change of variables $\gamma = \gamma_2^{-1} g$ is equal to the function 
 \begin{equation}
\label{eq-first-triangle-for-bernstein2}
g \longmapsto   \int_G \int _{\plus{N}} f_0(\gamma n) h_1 (g \gamma^{-1} )\,  \dd \gamma \dd n  .
\end{equation}
To proceed further we shall use the following simple  integration formula:
\begin{lemma} 
\label{lem-int-fmla}
If $\varphi$ is any integrable function on $G$, then 
\[
\int _G \varphi (\gamma)\, \dd \gamma   = \int_{\op{N}} \int_{L} \int_{\plus{N}} \varphi (\op{n} \ell \plus{n} )\plus{\delta}(\ell) \,  \dd \op{n} \dd \ell \dd \plus{n} .
\]
\end{lemma}

\begin{proof}
See for example \cite[Proposition 8.4.5]{KnappBeyond}.
\end{proof}

Keeping in mind the definition of the function $f_0$ that appears in  \eqref{eq-first-triangle-for-bernstein2}, by applying Lemma~\ref{lem-int-fmla} and    carrying out the integration over $\op{N}$ in Lemma~\ref{lem-int-fmla} we find that 
 the function \eqref{eq-first-triangle-for-bernstein2} is expressible as 
 \begin{equation*}
g \longmapsto    \int_L  \int _{\plus{N}}\int _{\plus{N}} f(\ell )\plus{\delta}(\ell)^{1/2} \plus{v} ( \plus{n} n) h_1 (g \plus{n}^{-1} \ell^{-1}  )\,   \dd \ell\dd \plus{n}  \dd n   ,
\end{equation*}
and integrating over the $n$-variable gives
 \begin{equation*}
g \longmapsto   \int_L  \int _{\plus{N}} f(\ell )\plus{\delta}(\ell)^{1/2}   h_1 (g \plus{n}^{-1} \ell^{-1}  )\,   \dd \ell  \dd \plus{n}  ,
\end{equation*}
or equivalently 
 \begin{equation*}
g \longmapsto   \int_L  \int _{\plus{N}} f(\ell )\plus{\delta}(\ell)^{-1/2}   h_1 (g \ell^{-1} \plus{n}^{-1}   )\,   \dd \ell  \dd \plus{n}  .
\end{equation*}
Carrying out the integration over the $\plus{n}$ gives 
 \begin{equation}
\label{eq-first-triangle-for-bernstein6}
g \longmapsto   \int_L   f(\ell )\plus{\delta}(\ell)^{-1/2}   \plus{J}h_1 (g \ell^{-1}   )\,   \dd \ell    .
\end{equation}
But this    is precisely the value at $g\in G$ of the convolution $\plus{J}h_1\cdot f$, and since 
\[
\plus{J} h_1 = \plus {J} \plus{I} h = h,
\]
we obtain in \eqref{eq-first-triangle-for-bernstein6} the value at $g\in G$ of $ h \cdot f$, as required in \eqref{eq-first-triangle-relation}.

A similar computation, using the formula for $B$ in terms of $\op{I}$, as in item (b) of Theorem~\ref{I_theorem}, handles \eqref{eq-second-triangle-relation}.

\section{Fourier transforms and intertwiners}
\label{appendix-fourier}

In this final section we shall give a proof of Theorem~\ref{I_theorem} (which summarized the properties of the intertwining integrals $\pnm{J}$ that were needed to prove the second adjoint theorem).  The arguments   rely on substantial results  from  Harish-Chandra's theory of the Plancherel formula, and are rather technical (we shall not attempt a conceptual approach here). But once  again the fact that we are concentrating   on the group $G=\SL(2,\R)$  helps  simplify matters.  

The general idea is to analyze the intertwining integrals $\pnm{J}$ using the   Fourier transform.   Denote by ${N}\subseteq G$ either one of the subgroups $\pnm{N}$, and let $\mathfrak a$ be the Lie algebra of the positive diagonal subgroup $A\subseteq G$. 
For $h\in \HC(G/N)$ and $\mu\in \mathfrak a^*$ define 
\begin{equation}
\label{eq-fourier-def}
\widehat h (\mu) = \int _A (h{\cdot}a ) \, a^{-i\mu} \, \dd a ,
\end{equation}
 where  $a^{-i\mu}$ is shorthand for $e^{-i \mu(\log (a))}$.  The values of the integrand are smooth functions on $G/N$, and the integral converges absolutely  in  the Fr\'echet space $C^\infty (G/N)$ by virtue of the definition of the space $\HC(G/N)$.     Moreover 
 \[
 \widehat h (\mu) \in V(\mu) ,
 \]
 where 
\begin{equation}
\label{eqn-V-mu-def}
  V(\mu) = \left\{\phi\in C^\infty(G)\ |\ \phi(gan) = a^{-i\mu}  \delta(a)^{-1/2}  \phi(g) \right\}  .
\end{equation}
The Fr\'echet space $V(\mu)$ is a  parabolically induced representation of $G$; in classical terms it is a direct sum of the odd and the even principal series representation associated to $\mu$; thus
\begin{equation}
\label{eqn-V-mu-def2}
  V(\mu) = V(\mu)_{\even} \oplus V(\mu)_{\odd},
\end{equation}
 where the even and odd summands are characterized by the additional conditions 
 \[
 \phi(gz) = + \phi (g) \quad \text{and} \quad \phi (gz)  = - \phi (g),
 \]
 respectively, involving the element 
$
 z = \left [ \begin{smallmatrix} -1 & 0 \\ 0 & -1 \end{smallmatrix}\right ]$. 
 
Thanks to the Iwasawa decomposition the space $V(\mu)$ identifies with $\Sc(K)$ by restriction  of functions from $G$ to $K$.  After making this identification,  the Fourier transform defines a topological isomorphism 
\begin{equation}
\label{eqn-fourier-isomorphism}
\Fourier \colon \HC(G/N) \stackrel \cong \longrightarrow \Sc (\mathfrak a^*, \Sc(K)).
\end{equation}

In order to analyze the intertwining integrals 
\[
\pnm{J} \colon \HC(G/\mnp{N}) \longrightarrow C^\infty (G/\pnm{N})
\]
from the point of view of the Fourier transform, we need to know a bit more about functions in the range of $\pnm J$, so as to be able to apply the Fourier transform after $\pnm J$.   

\begin{lemma}
\label{lem-Jh-nearly-HC}
If $h\in \HC(G/\mnp{N})$, then 
\[
\bigl| (\pnm{J}h)(x) \bigr| \le \text{\rm constant}_h \cdot  \pnm{\delta}(x)^{-1/2} 
\]
for all $x\in G/\pnm{N}$.
\end{lemma}

\begin{proof}
This follows   from the estimate \eqref{eq-J-estimate} on the integrand defining $\pnm{J}h$.
\end{proof}

\begin{lemma}
\label{lem-half-very-rapid-decay}
If  $h\in \Sc(G/\op{N})$, then 
\[
 \bigl | (\plus{J}h)(x ) \bigr |  \le \text{\rm constant}_{h,m} \cdot  \plus\delta(x)^{m/2}
\]
for all $x\in G/\plus N$ and all $m\ge 1$.
\end{lemma}

\begin{proof}
Identify $G/\op{N}$ with the complement of the origin in the plane $\R^2$, as in   Remark~\ref{rem-GmodN-space}, where we noted that if  $h\in \Sc(G/\op{N})$, then for every $m\ge 0$ we have an estimate 
\[
| h(x)| \le  \text{\rm constant}_{h, m}  \cdot ( 1 + \| x\|^2 )^{-m}
\]
(along with further estimates at $x=0$ that won't concern us).  So if  
\[
a = \begin{bmatrix} e^t & 0 \\ 0 & e^{-t}\end{bmatrix}
\quad 
\text{and}
\quad 
\plus{n} = \begin{bmatrix} 1 & y \\ 0 & 1\end{bmatrix},
\]
then
\begin{equation*}
 \bigl |  h(ka\plus{n} ) \bigr | 
 \le 
 \text{\rm constant}_{h, m} \cdot (1+ y^2 e^{2t}+ e^{-2t} )^{-m} .
 \end{equation*}
Integrating the bound with respect to $y$, we find that 
\begin{equation*}
 \bigl | (\plus{J}h)(ka\plus{n} ) \bigr |  
   \le     \text{constant}_{h,m} \cdot \int _{-\infty}^\infty  (1+ e^{-2t} + y^2 e^{2t})^{-m}  \, \dd y  .
   \end{equation*}
   The integral is bounded by a constant depending on $m$, only, times  $e^{(2m-2)t}$, and the result follows.
\end{proof}
 
The lemma implies that if $h\in \Sc(G/\op{N})$, then the Fourier transform \eqref{eq-fourier-def} of $\plus{J} h$ (which exists in a    distributional sense thanks to Lemma~\ref{lem-Jh-nearly-HC}) extends   to a holomorphic function on the half-plane consisting of  those $\mu \in \mathfrak a ^*_\C$ with \emph{positive imaginary part} in the sense that 
\begin{equation}
\label{eq-half-plane2}
  \operatorname{Im}   \mu \left ( \left [ \begin{smallmatrix} 1 & 0 \\ 0 & -1 \end{smallmatrix} \right ] \right )   > 0 ,
\end{equation}
or equivalently 
\begin{equation}
\label{eq-half-plane}
\plus\delta(a)>1  \quad \Rightarrow \quad \left | a^{-i \mu}\right |  >  1 .
\end{equation}
Here are the details.  Denote by $\pnm{V}(\mu)$   the two versions of the $\SF$-modules $V(\mu)$ from \eqref{eqn-V-mu-def} associated to the two unipotent subgroups $\pnm{N}$.   

\begin{lemma}
Assume that  $  \operatorname{Im}   \mu  > 0$. If $h\in \Sc(G/\op{N})$, and if   $g\in G $, then the integral 
 \[
\widehat{\plus{J}h}(\mu)(g)  =  \int _A (\plus{J}h{\cdot}a )(g) \, a^{-i\mu} \, \dd a \]
 converges absolutely  and defines a holomorphic function of $\mu$. For a fixed $\mu$ with positive imaginary part, the same integral defines  an element of $\plus{V}(\mu)$.
 \end{lemma}
 
\begin{proof} 
According to the definitions, 
 \[
  \widehat{\plus{J}h}(\mu)(x)  
	 = \int_A (\plus{J}h) (xa^{-1}) \plus\delta(a)^{-1/2} a^{-i\mu} \, \dd a .
\]

 The integral converges in the region  $\plus\delta(a)>1$ thanks to Lemma~\ref{lem-half-very-rapid-decay}, since the very rapid decay of $a\mapsto (\plus{J}h) (xa^{-1}) $ as $\plus \delta(a)\to + \infty$ compensates for the exponential growth of $ a^{-i\mu}$.  The integral  converges    in the region $\plus\delta(a)< 1$ thanks to Lemma~\ref{lem-Jh-nearly-HC}, which tells us that $(\plus{J}h) (xa^{-1}) \plus\delta( a)^{-1/2} $ is bounded,  while by  \eqref{eq-half-plane2} the term $a^{-i\mu}$ decays exponentially as $\plus\delta(a)\to -\infty$.  Smoothness in $g\in G$ follows, as usual, by differentiating under the integral sign, and the fact that $\plus{J}h\in \plus{V}(\mu)$ follows from a change of variables.
\end{proof}

In addition, we can apply the Fourier transform first, and then apply $\plus{J}$:

\begin{lemma}
If $\operatorname{Im}(\mu)> 0$ as in \eqref{eq-half-plane2}, and if $h\in \op{V}(\mu) $, then the integral 
\[
\plus{J}h (g) = \int _{\plus{N}} h(gn)\,\dd n
\]
converges absolutely and defines a   morphism
\[
\plus{J}(\mu) \colon \op{V}(\mu) \longrightarrow \plus{V}(\mu)
\]
of $\SF$-modules over $\Sc(G)$.
\end{lemma} 

\begin{proof} 
The proof is essentially the same as the proof  of Lemma \ref{lem-half-very-rapid-decay}.
\end{proof}
 
The convergence of the integrals involved implies that  if $h\in \Sc(G/\op{N})$, and if $\operatorname{Im}(\mu)> 0$, then 
\begin{equation}
\label{eq-pointwise-J-def}
\widehat { \plus{J} h} (\mu) = \plus{J}(\mu) \widehat h (\mu).
\end{equation}
 This is the formula we are seeking, since we can now appeal to the known, detailed formulas for $\plus{J}(\mu)$ to proceed.

\begin{remark}
There is of course an exactly similar sequence of results for the other intertwining integral,
\[
\op{J} \colon \HC(G/\plus{N})\longrightarrow  C^\infty (G/\op{N}),
\]
involving now $\mu \in \mathfrak a^* _\C$ with $\operatorname{Im} \mu < 0$.  We shall use these below.
\end{remark}

Denote by $\sigma_j\colon K \to U(1)$  the  continuous character defined by 
\[
\sigma_j \colon \left [\begin{smallmatrix} \cos(u)\,\, &  -\sin(u) \\ \sin (u)\,\, & \phantom{-}\cos(u)\end{smallmatrix}\right ] \longmapsto e^{\sqrt{-1}ju}.
\]
In the following proposition we identify $\pnm{V}(\mu)$ with $\Sc(K)$ by restriction of functions from $G$ to $K$.  Denote by 
\[
\Sc_j(K) \subseteq \Sc(K)
\]
the ({one-dimensional}) $\sigma_j$-isotypical subspace under the left translation action of $K$.  

\begin{remark}
Note that the decompositions of $\pnm{V}(\mu)$ into principal series representations corresponds to the decomposition 
\[
\Sc(K) = \Sc(K)_{\even} \oplus \Sc(K)_{\odd},
\]
where $ \Sc(K)_{\even/\odd}$ is the closed span of all $\Sc_j(K)$ with $j$ even/odd. 
\end{remark}

We shall also  identify $\mathfrak a^* _\C$ with $\C$ by evaluation on $  \left [ \begin{smallmatrix} 1 & 0 \\ 0 & -1 \end{smallmatrix} \right ] $.

\begin{theorem}\label{c_theorem0}
There are \textup{(}unique\textup{)} meromorphic functions 
\[
\pnm {c}{}^{(j)}: \germ a^*_{\C}  \longrightarrow \C
\] 
such that  the operator
\[
\plus{J}(\mu) \colon \Sc_j (K) \longrightarrow  \Sc_j (K)
\]
acts as multiplication by $\plus {c}{}^{(j)}(\mu)$ whenever $\operatorname{Im}\mu >0$, while the operator 
\[
\op{J}(\mu) \colon \Sc_j (K) \longrightarrow  \Sc_j (K)
\]
acts as multiplication by $\op {c}{}^{(j)}(\mu)$ whenever $\operatorname{Im}\mu <0$. \end{theorem}

\begin{proof}
The functions $\pnm{c}{}^{(j)}$ are computed directly in representation theory, and one finds  that
\begin{equation}\label{cj_equation}
\pnm c{}^{(j)}(\mu)  =  \frac{ \pi^{\frac{1}{2}} \Gamma\left( \frac{\mp i\mu}{2} \right) \Gamma\left( \frac{1\mp i\mu}{2} \right) }{ \Gamma \left( \frac{1\mp i\mu+j}{2} \right) \Gamma \left( \frac{ 1\mp i\mu-j}{2} \right) }.
\end{equation}
See \cite[Lemma 10.5.1]{Wallach2}, where an integral formula for $\pnm{c}{}^{(j)}$ is derived; and \cite[A.3]{Eguchi-Tanaka}, where the integrals arising in the case of $\SL(2,\R)$ are evaluated. Compare \cite[Theorem 31, p.252]{Varadarajan}.
  \end{proof}
  
For odd $j\in \mathbb Z $ the functions $\pnm{c}{}^{(j)}$ are   smooth on the real line $\mathfrak a^* \subseteq \mathfrak a^*_\C$, and they multiply the Schwartz space $\Sc(\mathfrak a^*)$ into itself. 

For even $j\in \mathbb \Z$ this is no longer true, since $\pnm{c}{}^{(j)}(\mu)$ has a pole at $\mu =0$.  But the product  $\mu \cdot \pnm{c}{}^{(j)}(\mu)$  multiplies the Schwartz space into itself.  So $\pnm{c}{}^{(j)}$ maps Schwartz functions into tempered distributions (namely products of Schwartz functions times the principal value distribution $1/\mu$).  Theorem~\ref{c_theorem0}  gives  the following formula for $\pnm{J}$:

  \begin{theorem}
  \label{c_theorem1}
For every $j \in \mathbb Z$ the  diagram
 \[
 \xymatrix{
 \Sc(\mathfrak a^*, \Sc_j(K)) \ar[r]^{ \pnm{c}{}^{(j)} } &   \Sc'(\mathfrak a^*, \Sc_j(K))  \\ 
 \HC(G/\pnm{N}) \ar[u]^{\Fourier}_\cong\ar[r]_{\pnm{J}}& \HC' (G/\mnp{N})_j\ar[u]_{\Fourier}^\cong
 }
 \]
is commutative \textup{(}the primes denote dual spaces of distributions\textup{)}. \qed
\end{theorem}

Obviously, the theorem suggests we invert $\pnm{J}$ by forming the reciprocals of the functions $\pnm{c}{}^{(j)}$.  For odd $j\in \mathbb Z $ the functions $\pnm{c}{}^{(j)}$ are    nowhere vanishing on $\mathfrak a^*$,  and their reciprocals do indeed multiply the Schwartz space into itself.    For even $j\in \mathbb Z$ the    functions $\pnm{c}{}^{(j)}$ have no zeros in $\mathfrak a^*$, and once again the reciprocal of $\pnm{c}{}^{(j)}$ is a multiplier of the Schwartz space.  More is true:
 
   \begin{theorem}\label{c_theorem2}
The combined operators  on the algebraic direct sum 
\[
 \bigoplus _{j\in \mathbb Z}  \Sc(\mathfrak a^*, \Sc_j(K)) \subseteq \Sc(\mathfrak a^*, \Sc(K))
\]
 that multiply  $\Sc_j (K)$-valued functions by   $\pnm{c}{}^{(j)}(\mu)^{-1}$ extend to   continuous operators
\[
\pnm c^{-1}: \Sc (\mathfrak a^* , \Sc(K)) \longrightarrow \Sc(\mathfrak a^*, \Sc(K)).
\]
\end{theorem}

\begin{proof}
This follows from the explicit formula \eqref{cj_equation} and the functional equation for the $\Gamma$-function.
\end{proof}

We have found operators satisfying item (a) of Theorem~\ref{I_theorem}: 
  
\begin{theorem}
 The operators $\pnm{I}$ defined by the diagram 
 \[
 \xymatrix{
 \Sc(\mathfrak a^*, \Sc(K)) \ar[r]^{ \pnm{c}^{-1}} &   \Sc(\mathfrak a^*, \Sc(K))  \\ 
 \HC(G/\pnm{N}) \ar[u]^{\Fourier}_\cong\ar[r]_{\pnm{I}}& \HC (G/\mnp{N})\ar[u]_{\Fourier}^\cong
 }
 \]
 are right-inverse to the intertwining operators $\pnm{J}$. \qed
\end{theorem}

Let us turn to item (b) in Theorem~\ref{I_theorem}.  It is an immediate consequence of the following result and  the fact that, thanks to Plancherel's formula, the Fourier isomorphism \eqref{eqn-fourier-isomorphism} is a unitary isomorphism for the obvious $L^2$-inner products.

   \begin{theorem}\label{c_theorem3}
   If $\mu \in \mathfrak a^*_\C$, then 
 $\overline{\op{c}{}^{(j)}(\mu)} =  {\plus{c}{}^{(j)}( \overline \mu)}$.\end{theorem}

\begin{proof} 
This may be verified directly from \eqref{cj_equation}.  It is also equivalent to a well-known adjoint relation among the intertwining operators $\pnm{J}(\mu)$; see \cite[Section 10.5.6]{Wallach2}
\end{proof}

 It remains to consider item (c), on wave packets.  Using the Fourier isomorphism \eqref{eqn-fourier-isomorphism} the operator $B$ from Theorem~\ref{I_theorem} can be viewed as the operator
\begin{equation*}
\widehat B \colon  \Sc(\mathfrak a^*, \Sc(K)) \otimes  \Sc(\mathfrak a^*, \Sc(K)) \longrightarrow C^\infty (G)
\end{equation*}
defined by the wave packet formula 
\begin{equation}
\label{eq-fourier-pic-B}
f_1 \otimes f_2 \longmapsto \left [g \mapsto \int _{\mathfrak a^*} \langle f_1(\mu),  \plus{c}^{-1}(\mu)g^{-1}f_2(\mu)\rangle _{L^2 (K)}\, \dd \mu \right ] ,
\end{equation}
where $g^{-1}$ acts on $f_2 (\mu)\in \Sc(K)$ via the restriction isomorphism
\[
\plus V(\mu) \stackrel \cong \longrightarrow \Sc(K).
\]
We need to show that \eqref{eq-fourier-pic-B} defines a function on $G$ that belongs to Harish-Chandra's Schwartz class. 

We shall use Harish-Chandra's wave packet theorem, and for this purpose it is convenient to consider separately the summands in the decomposition 
\[
 \Sc(\mathfrak a^*, \Sc(K)) =  \Sc(\mathfrak a^*, \Sc(K)_{\even}) \oplus  \Sc(\mathfrak a^*, \Sc(K)_{\odd}).
 \]
 Note that the summands are invariant under the action of $G$ and orthogonal, so we need only consider   the two cases 
 \begin{equation}
 \label{eq-even-or-odd}
 f_1,f_2\in  \Sc(\mathfrak a^*, \Sc(K)_{\even}) 
 \quad\text{or}\quad 
 f_1,f_2 \in \Sc(\mathfrak a^*, \Sc(K)_{\odd}) .
 \end{equation}
 in  \eqref{eq-fourier-pic-B}. In either case, Harish-Chandra showed that the function 
 \[
 g \mapsto \int _{\mathfrak a^*} \langle f_1(\mu),   g^{-1}f_2(\mu)\rangle _{L^2 (K)}\,  \alpha(\mu) \dd \mu
 \]
 belongs to $\HC(G)$, where $\alpha$ is the Plancherel density function for either the even or odd principal series, according to the two alternatives in \eqref{eq-even-or-odd}.   See \cite[Theorem 33, p.255]{Varadarajan} for the case of $G=\SL(2,\R)$ that concerns us here.
 
 In the odd case, the Plancherel density is a smooth and nowhere vanishing function, and both it and its reciprocal multiply the Schwartz space into itself. So since we can write 
 \begin{multline}\label{eq-plancherel-even}
   \int _{\mathfrak a^*} \langle f_1(\mu),   g^{-1}f_2(\mu)\rangle _{L^2 (K)}\, \dd \mu
 \\  =
   \int _{\mathfrak a^*} \langle f_1(\mu),   g^{-1}\alpha (\mu)^{-1}f_2(\mu)\rangle _{L^2 (K)}\, \alpha (\mu) \dd \mu,
\end{multline}
  we find that the left-hand side of \eqref{eq-plancherel-even} is a Harish-Chandra function of $g\in G$ for \emph{any} $f_1$ and $f_2$.
   
  The even case requires a bit more work, since the Plancherel density function vanishes at $\mu =0$.  But  $\mu^2 \alpha(\mu)^{-1}$  is smooth and multiplies the Schwartz space to itself, and so we find by the above argument that 
  
  \begin{theorem}
  \label{thm-HC-wave}
   If   $f_1,f_2\in \Sc(\mathfrak a^*, \Sc(K)_\even)$, and if both functions vanish at $\mu = 0$, or if one of the functions vanishes at $\mu =0$ to order two, then the integral 
 \[
     \int _{\mathfrak a^*} \langle f_1(\mu),   g^{-1}f_2(\mu)\rangle _{L^2 (K)}\, \dd \mu
 \]
  defines  a Harish-Chandra function of $g\in G$.   \qed 
  \end{theorem}
  In order to make use of this result, we can begin by noting that the function  
  \[
  \plus{c}^{-1} f_2 \in \Sc(\mathfrak a^*, \Sc(K)_{\even})
  \]
  does indeed vanish at $\mu = 0$, thanks to the explicit formula for $\plus{c}$.  But of course $f_2$ need not, and to cope with this circumstance we need to decompose the Schwartz space  $\Sc(\mathfrak a^*, \Sc(K)_{\even}) $ into even and odd parts using normalized intertwining operators, as follows.
  
\begin{lemma}
There is a $G$-equivariant involution 
\[
W \colon   \Sc(\mathfrak a^*, \Sc(K)_{\even}) \longrightarrow \Sc(\mathfrak a^*, \Sc(K)_{\even})
\]
of the form 
\[
(Wf)(\mu) = W(\mu) f(-\mu)
\]
where each $W(\mu)$ is a unitary operator on $\Sc(K)_\even$ with respect to the $L^2$-inner product, and where $W(0)$ is the identity operator.
\end{lemma}
  
 \begin{proof}
 An explicit construction of the normalized intertwiner in the special case we are considering is given in \cite[Proposition 22, p.243]{Varadarajan}.
 \end{proof}
 
 We can now decompose any $f\in \Sc(\mathfrak a^*, \Sc(K)_{\even}) $ into symmetric or antisymmetric parts under the action of the involution $W$ (that is, $+1$ or $-1$ eigenvectors).  
 
 \begin{lemma}
 \label{lem-symmetry-trick}
If one of $f_1, f_2\in \Sc( \mathfrak a^*, \Sc(K)_\even)$ is symmetric, and the other is antisymmetric, then the function 
 \[
   g \longmapsto   \int _{\mathfrak a^*} \langle f_1(\mu),   g^{-1}f_2(\mu)\rangle _{L^2 (K)}\, \dd \mu
 \]
 is identically zero.
 \end{lemma}

\begin{proof} 
It follows from unitarity and $G$-equivariance of the involution that 
 \[
 \int _{\mathfrak a^*} \langle (Wf_1)(\mu),   g^{-1}f_2(\mu)\rangle _{L^2 (K)}\, \dd \mu
 =
 \int _{\mathfrak a^*} \langle f_1(\mu),   g^{-1}(Wf_2)(\mu)\rangle _{L^2 (K)}\, \dd \mu ,
 \]
 and the lemma follows from this.
 \end{proof}
  
Returning to \eqref{eq-fourier-pic-B} ,  decompose $\plus{c}^{-1}f_2$ into its symmetric  and antisymmetric parts. Since $\plus{c}^{-1}f_2$ vanishes at $\mu =0$, its odd part vanishes at $\mu=0$  too, while its even part vanishes there  to \emph{second} order. 
Treating each separately we find from Lemma~\ref{lem-symmetry-trick} that  \eqref{eq-fourier-pic-B} is a sum of two terms, one from a pairing of two anti-symmetric functions, and one from the pairing of a symmetric function with another that vanishes to order two at $\mu=0$.  Item (c) in Theorem~\ref{I_theorem}  now follows from Theorem~\ref{thm-HC-wave}.

\bibliography{temperedbib}
\bibliographystyle{alpha}

\end{document}